\documentclass{amsart}

\usepackage{hyperref}
\usepackage{amsmath}
\usepackage{amssymb}
\usepackage{array}
\newtheorem{theorem}{Theorem}[section]

\newtheorem{proposition}[theorem]{Proposition}

\theoremstyle{definition}

\newtheorem{definition}[theorem]{Definition}

\theoremstyle{remark}

\newtheorem{remark}[theorem]{Remark}
\def\D{\displaystyle}
\numberwithin{equation}{section}

\DeclareMathOperator{\ord}{ord} \DeclareMathOperator{\Card}{Card}
\DeclareMathOperator{\trdeg}{tr.deg} \DeclareMathOperator{\rk}{rk}
\DeclareMathOperator{\Ker}{Ker}

\begin{document}

\title[Strength of difference schemes for reaction-diffusion equations]
{Evaluation of the Einstein's strength of difference schemes for some chemical reaction-diffusion equations}

\author{Alexander Evgrafov}
\address{Department of Analytical, Physical and Colloid Chemistry,
Sechenov First Moscow State Medical University, Moscow 119991,
Russia} \email{afkx\_farm@mail.ru}

\author{Alexander Levin}
\address{Department of Mathematics, The Catholic University of
America, Washington, D.C. 20064}

\email{levin@cua.edu}


\date\today


\begin{abstract}
In this paper we present a difference algebraic technique for the
evaluation of the Einstein's strength of a system of partial
difference equations and apply this technique to the comparative
analysis of difference schemes for chemical reaction-diffusion
equations. In particular, we analyze finite-difference schemes for
the Murray, Fisher, Burgers and some other reaction-diffusion
equations, as well as mass balance PDEs of chromatography from the
point of view of their strength.
\end{abstract}

\maketitle
\section{Introduction}

The concept of strength of a system of partial differential
equations (PDEs) was introduced by A. Einstein  as a characteristic
that provides a measure for the size of the solution space of such a
system. In \cite{Einstein} A. Einstein defined the strength of a
system of partial differential equations governing a physical field
as follows: ''\ldots the system of equations is to be chosen so that
the field quantities are determined as strongly as possible. In
order to apply this principle, we propose a method which gives a
measure of strength of an equation system. We expand the field
variables, in the neighborhood of a point $\mathcal{P}$, into a
Taylor series (which presupposes the analytic character of the
field); the coefficients of these series, which are the derivatives
of the field variables at $\mathcal{P}$, fall into sets according to
the degree of differentiation. In every such degree there appear,
for the first time, a set of coefficients which would be free for
arbitrary choice if it were not that the field must satisfy a system
of differential equations.  Through this system of differential
equations (and its derivatives with respect to the coordinates) the
number of coefficients is restricted, so that in each degree a
smaller number of coefficients is left free for arbitrary choice.
The set of numbers of 'free' coefficients for all degrees of
differentiation is then a measure of the 'weakness' of the system of
equations, and through this, also of its 'strength'.''

The comparison analysis of the strength of systems of PDEs
describing different mathematical models of the same process allows
one to obtain an optimal model, that is, to minimize the
''arbitrariness'' of the corresponding solutions. As an application
of this approach, A. Einstein determined that the potential and
field formulations of Maxwell equations have different strengths for
the dimension four. However, he did not obtain the exact expression
of the above-mentioned number of free coefficients as a function of
the degree of differentiation. The reason was that there were no
methods for the strength computation, so Einstein had to evaluate
this characteristic ''by hand''.

Even though there are a number of works on the strength of a system
of partial differential equations (in particular, on its relation to
Cartan characters), see, for example, \cite{Mariwalla},
\cite{Matthews1}, \cite{Matthews2}, \cite{Schutz}, \cite{Seiler1},
\cite{Seiler2}, \cite{Seiler3} and \cite{Sue}, there was no method
of its evaluation until 1980 when A. Mikhalev and E. Pankratev
\cite{MP} showed that the strength of a system of algebraic PDEs
(that is, a system of the form $f_{i} = 0$, $i\in I$, where $f_{i}$
are multivariate polynomials in unknown functions and their partial
derivatives) is expressed by the associated differential dimension
polynomial introduced by E. Kolchin \cite{Kolchin1} (see also
\cite[Chapter II, Theorem 6]{Kolchin2}). This observation has led to
algorithmic algebraic methods of computation of the strength of a
system of algebraic PDEs via computing the differential dimension
polynomial of the corresponding differential ideal in the algebra of
differential polynomials. The theoretical base for these methods and
their detailed description can be found in \cite{KLMP}.

Note that A. Einstein \cite{Einstein}, K. Mariwalla
\cite{Mariwalla}, M. Sue \cite{Sue} and some other authors who
investigated the concept of strength characterized the strength of a
system by the ''coefficient of freedom'', an integer, that is fully
determined by the leading coefficient of the differential dimension
polynomial. The fact that such a polynomial provides a far more
precise description of the strength than its leading term was
justified by the result of W. Sit \cite{Sit} who proved that the set
of differential dimension polynomials is well-ordered with respect
to the natural order ($f(t) < g(t)$ if and only if $f(r) < g(r)$ for
all sufficiently large integers $r$); this result  allows one to
distinguish two systems of PDEs with the same ''coefficient of
freedom'' by their strength.

Since 1980s the technique of dimension polynomials has been extended
to the analysis of systems of algebraic difference and
difference-differential equations. In a series of works whose
results are summarized in \cite{Levin1} the second author proved the
existence and developed some methods of computation of dimension
polynomials of difference field extensions and systems of algebraic
difference equations. These polynomials determine A. Einstein's
strength of a system of algebraic partial difference equations (we
give the details in Section 3 of this work) and, in particular,
allow one to evaluate the quality of difference schemes for PDEs
from the point of view of their strength.

In this paper we present a method of characteristic sets for
inversive difference polynomials and consider applications of this
technique to the analysis of difference schemes for
reaction-diffusion PDEs. We determine the strengths of systems of
partial difference equations that arise from such schemes and
perform their comparative analysis. We also perform a similar
investigation of difference schemes for mass balance PDEs of
chromatography that is one of the main methods for accurate, and
rapid determination of biologically active organic carboxylic acids
in objects such as infusion solutions and blood preservatives, see
\cite{Evgrafof0}, \cite{Evgrafof1} and \cite{Evgrafof2}.

\section{Preliminaries}
Throughout the paper, $\mathbb{N}, \mathbb{Z}$, $\mathbb{Q}$, and
$\mathbb{R}$ denote the sets of all non-negative integers, integers,
rational numbers, and real numbers, respectively. The number of
elements of a set $A$ is denoted by $\Card A$. As usual, $\mathbb
{Q}[t]$ denotes the ring of polynomials in one variable $t$ with
rational coefficients. By a ring we always mean an associative ring
with unity. All fields considered in the paper are supposed to be of
zero characteristic. Every ring homomorphism is unitary (maps unity
onto unity), every subring of a ring contains the unity of the ring.

If $B = A_{1}\times\dots\times A_{k}$ is a Cartesian product of $k$
ordered sets with orders $\leq_{1},\dots \leq_{k}$, respectively
($k\in \mathbb{N}$, $k\geq 1$), then by the product order on $B$ we
mean a partial order $\leq_{P}$ such that $(a_{1},\dots,
a_{k})\leq_{P}(a'_{1},\dots, a'_{k})$ if and only if
$a_{i}\leq_{i}a'_{i}$ for $i=1,\dots, k$. In particular, if $a =
(a_{1},\dots, a_{k}), \,a'=(a'_{1},\dots, a'_{k})\in
\mathbb{N}^{k}$, then $a\leq_{P}a'$ if and only if $a_{i}\leq
a'_{i}$ for $i=1,\dots, k$. We write $a <_{P}a'$ if $a\leq_{P}a'$
and $a\neq a'$.

\smallskip

In this section we present some background material needed for the
rest of the paper. In particular, we discuss basic concepts and
results of difference algebra and properties of dimension
polynomials associated with subsets of $\mathbb{N}^{m}$ and
$\mathbb{Z}^{m}$.

\bigskip

{\bf 2.1.\, Basic notions of difference algebra.}\, A {\em
difference ring} is a commutative ring $R$ together with a finite
set $\sigma = \{\alpha_{1}, \dots, \alpha_{m}\}$ of mutually
commuting injective endomorphisms of $R$ into itself. The set
$\sigma$ is called the {\em basic set\/} of the difference ring $R$,
and the endomorphisms $\alpha_{1}, \dots, \alpha_{m}$ are called
{\em translations.}\, A difference ring with a basic set $\sigma$ is
also called a {\em $\sigma$-ring}. If $\alpha_{1}, \dots,
\alpha_{m}$ are automorphisms of $R$, we say that $R$ is an {\em
inversive difference ring\/} with the basic set $\sigma$.  In this
case we denote the set $\{\alpha_{1}, \dots, \alpha_{m},
\alpha_{1}^{-1}, \dots, \alpha_{m}^{-1}\}$ by $\sigma^{\ast}$ and
call $R$ a $\sigma^{\ast}$-ring. If a difference ($\sigma$-) ring
$R$ is a field, it is called a difference (or $\sigma$-) field. If
$R$ is inversive, it is called an inversive difference field or a
$\sigma^{\ast}$-field.

Let $R$ be a difference (inversive difference) ring with a basic set
$\sigma$ and $R_{0}$ a subring of $R$ such that
$\alpha(R_{0})\subseteq R_{0}$ for any $\alpha \in \sigma$
(respectively, for any $\alpha \in \sigma^{\ast}$). Then $R_{0}$ is
called a {\em difference\/} or $\sigma$- (respectively, {\em
inversive difference} or $\sigma^{\ast}$-) {\em subring\/} of $R$,
while the ring $R$ is said to be a {\em difference} or $\sigma$-
(respectively, {\em inversive difference} or $\sigma^{\ast}$-) {\em
overring} of $R_{0}$. In this case the restriction of an
endomorphism $\alpha_{i}$ on $R_{0}$ is denoted by the same symbol
$\alpha_{i}$. If $R$ is a difference ($\sigma$-) or an inversive
difference ($\sigma^{\ast}$-) field and $R_{0}$ a subfield of $R$,
which is also a $\sigma$- \,(respectively, $\sigma^{\ast}$-) subring
of $R$, then  $R_{0}$ is said to be a $\sigma$-\, (respectively,
$\sigma^{\ast}$-) subfield of $R$; $R$, in turn, is called a
difference or $\sigma$- (respectively, inversive difference or
$\sigma^{\ast}$-) field extension or a $\sigma$-(respectively,
$\sigma^{\ast}$-) overfield of $R_{0}$. In this case we also say
that we have a $\sigma$- \,(or $\sigma^{\ast}$-) field extension
$R/R_{0}$.

If $R$ is a difference ring with a basic set $\sigma$ and $J$ is an
ideal of the ring $R$ such that $\alpha(J)\subseteq J$ for any
$\alpha \in \sigma$, then $J$ is called a {\em difference\/} (or
$\sigma$-) ideal of $R$. If a prime (maximal) ideal $P$ of $R$ is
closed with respect to $\sigma$ (that is, $\alpha(P)\subseteq P$ for
any $\alpha \in \sigma$), it is called a {\em prime} (respectively,
{\em maximal}) {\em difference} (or $\sigma$-) {\em ideal \/} of the
$\sigma$-ring $R$.

A $\sigma$-ideal $J$ of a $\sigma$-ring $R$ is called {\em
reflexive} (or a $\sigma^{\ast}$-ideal) if for any translation
$\alpha$, the inclusion $\alpha(a)\in J$ ($a\in R$) implies $a\in
J$. If $R$ is inversive, then a reflexive $\sigma$-ideal (that is,
an ideal $J$ such that $\alpha(J) = J$ for any $\alpha\in
\sigma^{\ast}$) is also called a $\sigma^{\ast}$-ideal.

\medskip

If $R$ is a difference ring with a basic set $\sigma = \{\alpha_{1},
\dots, \alpha_{n}\}$, then $T_{\sigma}$ (or $T$ if the set $\sigma$
is fixed) will denote the free commutative semigroup generated by
$\alpha_{1}, \dots, \alpha_{n}$. Elements of $T_{\sigma}$ will be
written in the multiplicative form $\alpha_{1}^{k_{1}}\dots
\alpha_{m}^{k_{m}}$ ($k_{1},\dots, k_{m}\in \mathbb{N}$) and
considered as injective endomorphisms of $R$ (which are the
corresponding compositions of the endomorphisms of $\sigma$). If the
$\sigma$-ring $R$ is inversive, then $\Gamma_{\sigma}$ (or $\Gamma$
if the set $\sigma$ is fixed) will denote the free commutative group
generated by the set $\sigma$. It is clear that elements of the
group $\Gamma_{\sigma}$ (written in the multiplicative form
$\alpha_{1}^{i_{1}}\dots \alpha_{m}^{i_{m}}$ with $i_{1},\dots,
i_{n}\in \mathbb{Z}$) act on $R$ as automorphisms and $T_{\sigma}$
is a subsemigroup of $\Gamma_{\sigma}$.

For any $a\in R$ and for any $\tau \in T_{\sigma}$, the element
$\tau(a)$ is called a {\em transform\/} of $a$. If the $\sigma$-ring
$R$ is inversive, then an element $\gamma(a)$ ($a\in R, \gamma \in
\Gamma_{\sigma}$) is also called a transform of $a$. An element
$a\in R$ is said to be a {\em constant} if $\alpha(a) = a$ for every
$\alpha\in\sigma$.

If $J$ is a $\sigma$-ideal of a $\sigma$-ring $R$, then $J^{\ast} =
\{a\in R\, |\, \tau(a)\in J$ for some $\tau \in T_{\sigma}\}$ is a
reflexive $\sigma$-ideal of $R$ contained in any reflexive
$\sigma$-ideal of $R$ containing $J$. The ideal $J^{\ast}$ is called
the {\em reflexive closure\/} of the $\sigma$-ideal $J$.

Let $R$ be a difference ring with a basic set $\sigma$ and
$S\subseteq R$. Then the intersection of all $\sigma$-ideals of $R$
containing $S$ is denoted by $[S]$. Clearly, $[S]$ is the smallest
$\sigma$-ideal of $R$ containing $S$; as an ideal, it is generated
by the set  $T_{\sigma}S = \{\tau(a) | \tau \in T_{\sigma}, a\in
S\}$. If $J = [S]$, we say that the $\sigma$-ideal $J$ is generated
by the set $S$ called a {\em set of $\sigma$-generators} of $J$. If
$S$ is finite, $S=\{a_{1},\dots, a_{k}\}$, we write $J =
[a_{1},\dots, a_{k}]$ and say that $J$ is a finitely generated
$\sigma$-ideal of the $\sigma$-ring $R$. (In this case elements
$a_{1},\dots, a_{k}$ are said to be $\sigma$-generators of $J$.)

If $R$ is an inversive difference ($\sigma$-) ring and $S\subseteq
R$, then the inverse closure of the $\sigma$-ideal $[S]$ is denoted
by $[S]^{\ast}$. It is easy to see that $[S]^{\ast}$ is the smallest
$\sigma^{\ast}$-ideal of $R$ containing $S$; as an ideal, it is
generated by the set $\Gamma_{\sigma}S = \{\gamma(a) | \gamma \in
\Gamma_{\sigma}, a\in S\}$. If $S$ is finite, $S=\{a_{1},\dots,
a_{k}\}$, we write $[a_{1},\dots, a_{k}]^{\ast}$ for $I =
[S]^{\ast}$ and say that $I$ is a finitely generated
$\sigma^{\ast}$-ideal of $R$. (In this case, elements $a_{1},\dots,
a_{k}$ are said to be $\sigma^{\ast}$-generators of $I$.)

Let $R$ be a difference ring with a basic set $\sigma$, $R_{0}$ a
$\sigma$-subring of $R$ and $B\subseteq R$. The intersection of all
$\sigma$-subrings of $R$ containing $R_{0}$ and $B$ is called the
{\em $\sigma$-subring of $R$ generated by the set $B$ over $R_{0}$},
it is denoted by $R_{0}\{B\}$. (As a ring, $R_{0}\{B\}$ coincides
with the ring $R_{0}[\{\tau(b) | b\in B, \tau \in T_{\sigma}\}]$
obtained by adjoining the set $\{\tau(b) | b\in B, \tau \in
T_{\sigma}\}$ to the ring $R_{0}$). The set $B$ is said to be the
set of {\em $\sigma$-generators} of the $\sigma$-ring $R_{0}\{B\}$
over $R_{0}$. If this set is finite, $B = \{b_{1},\dots, b_{k}\}$,
we say that $R' = R_{0}\{B\}$ is a finitely generated difference (or
$\sigma$-) ring extension (or overring) of $R_{0}$ and write $R' =
R_{0}\{b_{1},\dots, b_{k}\}$. If $R$ is a $\sigma$-field, $R_{0}$ a
$\sigma$-subfield of $R$ and $B\subseteq R$, then the intersection
of all $\sigma$-subfields of $R$ containing $R_{0}$ and $B$ is
denoted by $R_{0}\langle B\rangle$ (or $R_{0}\langle b_{1},\dots,
b_{k}\rangle$ if $B=\{b_{1},\dots, b_{k}\}$ is a finite set). This
is the smallest $\sigma$-subfield of $R$ containing $R_{0}$ and $B$;
it coincides with the field $R_{0}(\{\tau(b) | b\in B, \tau \in
T_{\sigma}\})$. The set $B$ is called a set of {\em
$\sigma$-generators} of the $\sigma$-field $R_{0}\langle B\rangle$
over $R_{0}$.

Let $R$ be an inversive difference ring with a basic set $\sigma$,
$R_{0}$ a $\sigma^{\ast}$-subring of $R$ and $B\subseteq R$. Then
the intersection of all $\sigma^{\ast}$-subrings of $R$ containing
$R_{0}$ and $B$ is the smallest $\sigma^{\ast}$-subring of $R$
containing $R_{0}$ and $B$. This ring coincides with the ring
$R_{0}[\{\gamma(b) | b\in B, \gamma \in \Gamma_{\sigma}\}]$; it is
denoted by $R_{0}\{B\}^{\ast}$. The set $B$ is said to be a {\em set
of $\sigma^{\ast}$-generators} of $R_{0}\{B\}^{\ast}$ over $R_{0}$.
If $B = \{b_{1},\dots, b_{k}\}$ is a finite set, we say that $S =
R_{0}\{B\}^{\ast}$ is a finitely generated inversive difference (or
$\sigma^{\ast}$-) ring extension (or overring) of $R_{0}$ and write
$S = R_{0}\{b_{1},\dots, b_{k}\}^{\ast}$.

If $R$ is a $\sigma^{\ast}$-field, $R_{0}$ a
$\sigma^{\ast}$-subfield of $R$ and $B\subseteq R$, then the
intersection of all $\sigma^{\ast}$-subfields of $R$ containing
$R_{0}$ and $B$ is denoted by $R_{0}\langle B\rangle^{\ast}$. This
is the smallest $\sigma^{\ast}$-subfield of $R$ containing $R_{0}$
and $B$; it coincides with the field $R_{0}(\{\gamma(b) | b\in B,
\gamma \in \Gamma_{\sigma}\})$. The set $B$ is called a {\em set of
$\sigma^{\ast}$-generators of the $\sigma^{\ast}$-field extension}
$R_{0}\langle B\rangle^{\ast}$ {\em of} $R_{0}$.  If $B$ is finite,
$B=\{b_{1},\dots, b_{k}\}$, we write $R_{0}\langle b_{1},\dots,
b_{k}\rangle^{\ast}$ for $R_{0}\langle B\rangle^{\ast}$.

In what follows we often consider two or more difference rings
$R_{1},\dots, R_{p}$ with the same basic set
$\sigma=\{\alpha_{1},\dots, \alpha_{m}\}$. Formally speaking, it
means that for every $i=1,\dots, p$, there is some fixed mapping
$\nu_{i}$ from the set $\sigma$ into the set of all injective
endomorphisms of the ring $R_{i}$ such that any two endomorphisms
$\nu_{i}(\alpha_{j})$ and $\nu_{i}(\alpha_{k})$ of $R_{i}$ commute
($1\leq j, k\leq n$). We shall identify elements $\alpha_{j}$ with
their images $\nu_{i}(\alpha_{j})$ and say that elements of the set
$\sigma$ act as mutually commuting injective endomorphisms of the
ring $R_{i}$ ($i=1,\dots, p$).

Let $R_{1}$ and $R_{2}$ be difference rings with the same basic set
$\sigma=\{\alpha_{1},\dots, \alpha_{m}\}$. A ring homomorphism
$\phi: R_{1}\rightarrow R_{2}$ is called a {\em difference\/} (or
$\sigma$-) {\em homomorphism\/} if $\phi(\alpha(a)) =
\alpha(\phi(a))$ for any $\alpha \in \sigma, a\in R_{1}$. Clearly,
if $\phi: R_{1} \rightarrow R_{2}$ is a $\sigma$-homomorphism of
inversive difference rings, then $\phi(\alpha^{-1}(a)) =
\alpha^{-1}(\phi(a))$ for any $\alpha \in \sigma, \, a\in R_{1}$. If
a $\sigma$-homomorphism is an isomorphism (endomorphism,
automorphism, etc.), it is called a difference (or $\sigma$-)
isomorphism (respectively, difference (or $\sigma$-) endomorphism,
difference (or $\sigma$-) automorphism, etc.).  If $R_{1}$ and
$R_{2}$ are two $\sigma$-overrings of the same $\sigma$-ring $R_{0}$
and $\phi: R_{1}\rightarrow R_{2}$ is a $\sigma$-homomorphism such
that $\phi(a) = a$ for any $a\in R_{0}$, we say that $\phi$ is a
difference (or $\sigma$-) homomorphism over $R_{0}$ or that $\phi$
leaves the ring $R_{0}$ fixed.

It is easy to see that the kernel of any $\sigma$-homomorphism of
$\sigma$-rings $\phi: R\rightarrow R'$ is an inversive
$\sigma$-ideal of $R$. Conversely, let $g$ be a surjective
homomorphism of a $\sigma$-ring $R$ onto a ring $S$ such that $\Ker
g$ is a $\sigma^{\ast}$-ideal of $R$. Then there is a unique
structure of a $\sigma$-ring on $S$ such that $g$ is a
$\sigma$-homomorphism. In particular, if $I$ is a
$\sigma^{\ast}$-ideal of a $\sigma$-ring $R$, then the factor ring
$R/I$ has a unique structure of a $\sigma$-ring such that the
canonical surjection $R\rightarrow R/I$ is a $\sigma$-homomorphism.
In this case $R/I$ is said to be the {\em difference} (or $\sigma$-)
{\em factor ring} of $R$ by the $\sigma^{\ast}$-ideal $I$.

If a difference (inversive difference) ring $R$ with a basic set
$\sigma$ is an integral domain, then its quotient field $Q(R)$ can
be naturally considered as a $\sigma$-(respectively,
$\sigma^{\ast}$-) overring of $R$. (We identify an element $a\in R$
with its canonical image $\frac{a}{1}$ in $Q(R)$.) In this case
$Q(R)$ is said to be the {\em quotient difference} (or quotient
$\sigma$-) {\em field} of $R$. Clearly, if the $\sigma$-ring $R$ is
inversive, then its quotient $\sigma$-field $Q(R)$ is also
inversive. Furthermore, if a $\sigma$-field $K$ contains an integral
domain $R$ as a $\sigma$-subring, then $K$ contains the quotient
$\sigma$-field $Q(R)$.\\

Let $R$ be a difference ring with a basic set $\sigma =
\{\alpha_{1}, \dots, \alpha_{m}\}$, $T_{\sigma}$ the free
commutative semigroup generated by $\sigma$, and $U = \{u_{\lambda}
| \lambda \in \Lambda\}$ a family of elements from some
$\sigma$-overring of $R$. We say that the family $U$ is {\em
transformally} (or $\sigma$-{\em algebraically) dependent} over $R$,
if the family $T_{\sigma}(U) = \{\tau(u_{\lambda}) | \tau \in
T_{\sigma}, \lambda \in \Lambda\}$ is algebraically dependent over
$R$ (that is, there exist elements $v_{1},\dots, v_{k}\in
T_{\sigma}(U)$ and a non-zero polynomial $f(X_{1},\dots, X_{k})$
with coefficients in $R$ such that $f(v_{1},\dots, v_{k}) = 0$).
Otherwise, the family $U$ is said to be {\em transformally} (or
$\sigma$-{\em algebraically) independent} over $R$ or a family of
{\em difference} (or $\sigma$-) {\em indeterminates} over $R$.  In
the last case, the $\sigma$-ring $R\{(u_{\lambda})_{\lambda \in
\Lambda}\}$ is called the {\em algebra of difference} (or $\sigma$-)
{\em polynomials} in the difference (or $\sigma$-) indeterminates
$\{(u_{\lambda})_{\lambda \in \Lambda}\}$ over $R$. If a family
consisting of one element $u$ is $\sigma$-algebraically dependent
over $R$, the element $u$ is said to be {\em transformally
algebraic} (or {\em $\sigma$-algebraic}) over the $\sigma$-ring $R$.
If the set $\{\tau(u) | \tau \in T_{\sigma}\}$ is algebraically
independent over $R$, we say that $u$ is {\em transformally} (or
$\sigma$-) {\em transcendental} over the ring $R$.

Let $R$ be a $\sigma$-field, $L$ a $\sigma$-overfield of $R$, and
$S\subseteq L$. We say that {\em the set $S$ is $\sigma$-algebraic
over $R$} if every element $a\in S$ is $\sigma$-algebraic over $R$.
If every element of the field $L$ is $\sigma$-algebraic over $R$, we
say that $L$ is a {\em $\sigma$-algebraic field extension} of the
$\sigma$-field $R$.

\begin{proposition}  {\em (\cite[Proposition 3.3.7]{KLMP})}.
Let $R$ be a difference ring with
a basic set $\sigma$ and $I$ an arbitrary set. Then there exists an
algebra of $\sigma$-polynomials over $R$ in a family of
$\sigma$-indeterminates with indices from the set $I$. If $S$ and
$S'$ are two such algebras, then there exists a $\sigma$-isomorphism
$S\rightarrow S'$ that leaves the ring $R$ fixed. If $R$ is an
integral domain, then any algebra of $\sigma$-polynomials over $R$
is an integral domain.
\end{proposition}

The algebra of $\sigma$-polynomials over the $\sigma$-ring $R$ can
be constructed as follows. Let $T = T_{\sigma}$ and let $S$ be the
polynomial $R$-algebra in the set of indeterminates $\{y_{i,
\tau}\}_{i\in I, \tau \in T}$ with indices from the set $I\times T$.
For any $f\in S$ and $\alpha \in \sigma$, let $\alpha(f)$ denote the
polynomial from $S$ obtained by replacing every indeterminate $y_{i,
\tau}$ that appears in $f$ by $y_{i, \alpha \tau}$ and every
coefficient $a\in R$ by $\alpha(a)$. We obtain an injective
endomorphism $S\rightarrow S$ that extends the original endomorphism
$\alpha$ of $R$ to the ring $S$ (this extension is denoted by the
same letter $\alpha$). Setting $y_{i}=y_{i, 1}$ (where $1$ denotes
the identity of the semigroup $T$) we obtain a set $\{y_{i} | i\in
I\}$ whose elements are $\sigma$-algebraically independent over $R$
and generate $S$ as a $\sigma$-ring extension of $R$. Thus, $S =
R\{(y_{i})_{i\in I}\}$ is an algebra of $\sigma$-polynomials over
$R$ in a family of $\sigma$-indeterminates $\{y_{i} | i\in I\}$.

\medskip

Let $R$ be an inversive difference ring with a basic set $\sigma$,
$\Gamma = \Gamma_{\sigma}$, $I$ a set, and $S^{\ast}$ a polynomial
ring in the set of indeterminates $\{y_{i, \gamma}\}_{i\in I, \gamma
\in \Gamma}$ with indices from the set $I\times \Gamma$. If we
extend the automorphisms $\beta \in \sigma^{\ast}$ to $S^{\ast}$
setting $\beta(y_{i, \gamma}) = y_{i, \beta \gamma}$ for any $y_{i,
\gamma}$ and denote $y_{i, 1}$ by $y_{i}$, then $S^{\ast}$ becomes
an inversive difference overring of $R$ generated (as a
$\sigma^{\ast}$-overring) by the family $\{(y_{i})_{i\in I}\}$.
Obviously, this family is {\em $\sigma^{\ast}$-algebraically
independent} over $R$, that is, the set $\{\gamma(y_{i}) | \gamma
\in \Gamma, i\in I\}$ is algebraically independent over $R$. (Note
that a set is $\sigma^{\ast}$-algebraically dependent (independent)
over an inversive $\sigma$-ring if and only if this set is
$\sigma$-algebraically dependent (respectively, independent) over
this ring.)  The ring $S^{\ast} = R\{(y_{i})_{i\in I}\}^{\ast}$ is
called the {\em algebra of inversive difference} (or
$\sigma^{\ast}$-) {\em polynomials} over $R$ in the set of
$\sigma^{\ast}$-indeterminates $\{(y_{i})_{i\in I}\}$. It is easy to
see that $S^{\ast}$ is the inversive closure of the ring of
$\sigma$-polynomials $S=R\{(y_{i})_{i\in I}\}$ over $R$ in the sense
that $S^{\ast}$ is the smallest inversive $\sigma$-overring of $S$
with the following property: for every $a\in S^{\ast}$, there exists
$\tau\in T_{\sigma}$ such that $\tau(a)\in S$. Furthermore, if a
family $\{(u_{i})_{i\in I}\}$ from some $\sigma^{\ast}$-overring of
$R$ is {\em $\sigma^{\ast}$-algebraically independent} over $R$,
then the inversive difference ring $R\{(u_{i})_{i\in I}\}^{\ast}$ is
naturally  $\sigma$-isomorphic to $S^{\ast}$. Any such overring
$R\{(u_{i})_{i\in I}\}^{\ast}$ is said to be an algebra of inversive
difference (or $\sigma^{\ast}$-) polynomials over $R$ in the set of
$\sigma^{\ast}$-indeterminates $\{(u_{i})_{i\in I}\}$. We obtain the
following analog of Proposition 2.1.

\begin{proposition} {\em (\cite[Proposition 3.4.4]{KLMP})}.
Let $R$ be an inversive difference ring with a basic set $\sigma$
and $I$ an arbitrary set. Then there exists an algebra of
$\sigma^{\ast}$-polynomials over $R$ in a family of
$\sigma^{\ast}$-indeterminates with indices from the set $I$. If $S$
and $S'$ are two such algebras, then there exists a
$\sigma^{\ast}$-isomorphism $S\rightarrow S'$ that leaves the ring
$R$ fixed. If $R$ is an integral domain, then any algebra of
$\sigma^{\ast}$-polynomials over $R$ is an integral domain.
\end{proposition}

Let $R$ be a $\sigma$-ring, $R\{(y_{i})_{i\in I}\}$ an algebra of
difference polynomials in a family of $\sigma$-indeterminates
$\{(y_{i})_{i\in I}\}$, and $\{(\eta_{i})_{i\in I}\}$  a set of
elements from some $\sigma$-overring of $R$. Since the set
$\{\tau(y_{i}) | i\in I, \tau \in T_{\sigma}\}$ is algebraically
independent over $R$, there exists a unique ring homomorphism
$\phi_{\eta}: R[\tau(y_{i})_{i\in I, \tau \in
T_{\sigma}}]\rightarrow R[\tau(\eta_{i})_{i\in I, \tau \in
T_{\sigma}}]$ that maps every $\tau(y_{i})$ onto $\tau(\eta_{i})$
and leaves $R$ fixed. Clearly, $\phi_{\eta}$ is a surjective
$\sigma$-homomorphism of $R\{(y_{i})_{i\in I}\}$ onto
$R\{(\eta_{i})_{i\in I}\}$; it is called the {\em substitution} of
$(\eta_{i})_{i\in I}$ for $(y_{i})_{i\in I}$. Similarly, if $R$ is
an inversive $\sigma$-ring, $R\{(y_{i})_{i\in I}\}^{\ast}$ an
algebra of $\sigma^{\ast}$-polynomials over $R$ and
$(\eta_{i})_{i\in I}$ a family of elements from a
$\sigma^{\ast}$-overring of $R$, one can define a surjective
$\sigma$-homomorphism $R\{(y_{i})_{i\in I}\}^{\ast}\rightarrow
R\{(\eta_{i})_{i\in I}\}^{\ast}$ that maps every $y_{i}$ onto
$\eta_{i}$ and leaves the ring $R$ fixed. This homomorphism is also
called the substitution of $(\eta_{i})_{i\in I}$ for $(y_{i})_{i\in
I}$. (It will be always clear whether we talk about substitutions
for difference or inversive difference polynomials.) If $g$ is a
$\sigma$- or $\sigma^{\ast}$- polynomial, then its image under a
substitution of $(\eta_{i})_{i\in I}$ for $(y_{i})_{i\in I}$ is
denoted by $g((\eta_{i})_{i\in I})$. The kernel of a substitution
$\phi_{\eta}$ is an inversive difference ideal of the $\sigma$-ring
$R\{(y_{i})_{i\in I}\}$ (or the $\sigma^{\ast}$-ring
$R\{(y_{i})_{i\in I}\}^{\ast}$); it is called the {\em defining
difference} (or $\sigma$-) {\em ideal} of the family
$(\eta_{i})_{i\in I}$ over $R$. If $R$ is a $\sigma$- (or
$\sigma^{\ast}$-) field and $(\eta_{i})_{i\in I}$ is a family of
elements from some $\sigma$- (respectively, $\sigma^{\ast}$-)
overfield $S$, then $R\{(\eta_{i})_{i\in I}\}$ (respectively,
$R\{(\eta_{i})_{i\in I}\}^{\ast}$) is an integral domain (it is
contained in the field $S$). It follows that the defining
$\sigma$-ideal $P$ of the family  $(\eta_{i})_{i\in I}$ over $R$ is
a reflexive prime difference ideal of the ring $R\{(y_{i})_{i\in
I}\}$ (respectively, of the ring of $\sigma^{\ast}$-polynomials
$R\{(y_{i})_{i\in I}\}^{\ast}$). Therefore, the difference field
$R\langle(\eta_{i})_{i\in I}\rangle$ can be treated as the quotient
$\sigma$-field of the $\sigma$-ring $R\{(y_{i})_{i\in I}\}/P$. (In
the case of inversive difference rings, the $\sigma^{\ast}$-field
$R\langle(\eta_{i})_{i\in I}\rangle^{\ast}$ can be considered as a
quotient $\sigma$-field of the $\sigma^{\ast}$-ring
$R\{(y_{i})_{i\in I}\}^{\ast}/P$.)

\medskip

Let $K$ be a difference field with a basic set $\sigma$ and $n$ a
positive integer. By an {\em $n$-tuple over} $K$ we mean an
$n$-dimensional vector $a = (a_{1},\dots, a_{n})$ whose coordinates
belong to some $\sigma$-overfield of $K$. If the $\sigma$-field $K$
is inversive, the coordinates of an $n$-tuple over $K$ are supposed
to lie in some $\sigma^{\ast}$-overfield of $K$. If each $a_{i}$
($1\leq i\leq n$) is $\sigma$-algebraic over the $\sigma$-field $K$,
we say that the {\em $n$-tuple $a$ is $\sigma$-algebraic over $K$}.

\begin{definition} Let $K$ be a difference (inversive
difference) field with a basic set $\sigma$ and let $R$ be the
algebra of $\sigma$- (respectively, $\sigma^{\ast}$-) polynomials in
finitely many $\sigma$- (respectively, $\sigma^{\ast}$-)
indeterminates $y_{1},\dots, y_{n}$ over $K$. Furthermore, let $\Phi
= \{f_{j} | j\in J\}$ be a set of $\sigma$- (respectively,
$\sigma^{\ast}$-) polynomials from $R$. An $n$-tuple $\eta =
(\eta_{1},\dots, \eta_{n})$ over $K$ is said to be a solution of the
set $\Phi$ or a solution of the system of difference algebraic
equations  $f_{j}(y_{1},\dots, y_{n}) = 0$ ($j\in J$) if $\Phi$ is
contained in the kernel of the substitution of $(\eta_{1},\dots,
\eta_{s})$ for $(y_{1},\dots, y_{n})$. In this case we also say that
$\eta$ annuls $\Phi$. (If $\Phi$ is a subset of a ring of inversive
difference polynomials, the system is said to be a system of
algebraic $\sigma^{\ast}$-equations.) A system of algebraic
difference equations $\Phi$ is called prime if the reflexive
difference ideal generated by $\Phi$ in the ring of $\sigma$ (or
$\sigma^{\ast}$-) polynomials is prime.
\end{definition}

As we have seen, if one fixes an $n$-tuple $\eta = (\eta_{1},\dots,
\eta_{n})$ over a $\sigma$-field $F$, then all $\sigma$-polynomials
of the ring $K\{y_{1},\dots, y_{n}\}$, for which $\eta$ is a
solution, form a reflexive prime difference ideal, the {\em defining
$\sigma$-ideal} of $\eta$. If $\eta$ is an $n$-tuple over a
$\sigma^{\ast}$-field $K$, then all $\sigma^{\ast}$-polynomials $g$
of the ring $K\{y_{1},\dots, y_{n}\}^{\ast}$ such that
$g(\eta_{1},\dots, \eta_{n}) = 0$ form a prime $\sigma^{\ast}$-ideal
of $K\{y_{1},\dots, y_{n}\}^{\ast}$ called the {\em defining
$\sigma^{\ast}$-ideal} of $\eta$ over $K$.

Let $\Phi$ be a subset of the algebra of $\sigma$-polynomials
$K\{y_{1},\dots, y_{n}\}$ over a $\sigma$-field $K$. An $n$-tuple
$\eta = (\eta_{1},\dots, \eta_{n})$ over $K$ is called a {\em
generic zero} of $\Phi$ if for any $\sigma$-polynomial $f\in
K\{y_{1},\dots, y_{n}\}$, the inclusion $f\in \Phi$ holds if and
only if $f(\eta_{1},\dots, \eta_{n})=0$. If the $\sigma$-field $K$
is inversive, then the notion of a generic zero of a subset of
$K\{y_{1},\dots, y_{n}\}^{\ast}$ is defined similarly.

Two $n$-tuples $\eta = (\eta_{1},\dots, \eta_{n})$ and $\zeta =
(\zeta_{1},\dots, \zeta_{n})$ over a $\sigma$- (or $\sigma^{\ast}$-)
field $K$ are called {\em equivalent} over $K$ if there is a
$\sigma$-homomorphism $K\langle \eta_{1},\dots, \eta_{n}\rangle
\rightarrow K\langle \zeta_{1},\dots, \zeta_{n}\rangle$
(respectively, $K\langle \eta_{1},\dots, \eta_{n}\rangle^{\ast}
\rightarrow K\langle \zeta_{1},\dots, \zeta_{n}\rangle^{\ast}$) that
maps each $\eta_{i}$ onto $\zeta_{i}$ and leaves the field $K$
fixed.

\begin{proposition}{\em (\cite[Proposition 3.3.7]{KLMP})}. Let $R$ denote the algebra of
$\sigma$-polynomials $K\{y_{1},\dots, y_{s}\}$ over a difference
 field $K$ with a basic set $\sigma$.

{\em (i)}\, A set $\Phi \subsetneqq R$ has a generic zero if and
only if $\Phi$ is a prime reflexive $\sigma$-ideal of $R$. If
$(\eta_{1},\dots, \eta_{n})$ is a generic zero of $\Phi$, then
$K\langle \eta_{1},\dots, \eta_{n}\rangle$ is $\sigma$-isomorphic to
the quotient $\sigma$-field of $R/\Phi$.

{\em (ii)}\, Any $n$-tuple over $K$ is a generic zero of some prime
reflexive $\sigma$-ideal of $R$.

{\em (iii)}\, If two $n$-tuples over $K$ are generic zeros of the
same prime reflexive $\sigma$-ideal of $R$, then these $n$-tuples
are equivalent.
\end{proposition}

{\bf 2.2. Numerical polynomials of subsets of $\mathbb{N}^{m}$ and
$\mathbb{Z}^{m}$}.
\begin{definition}
A polynomial $f(t)$ in one variable $t$ with rational coefficients is called
numerical if $f(r)\in \mathbb{Z}$ for all sufficiently large $r\in\mathbb{Z}$.
\end{definition}

Of course, every polynomial with integer coefficients is numerical.
As an example of a numerical polynomial with non-integer
coefficients one can consider a polynomial $\D{t\choose k}$ \, where
$k\in\mathbb{N}$. (As usual, $\D{t\choose k}$ ($k\geq 1$) denotes
the polynomial $\D\frac{t(t-1)\dots (t-k+1)}{k!}$, $\D{t\choose0} =
1$, and $\D{t\choose k} = 0$ if $k < 0$.)

The following theorem proved in ~\cite[Chapter 0, section 17]{Kolchin2} gives the ``canonical'' representation of a numerical polynomial.

\begin{theorem}
Let $f(t)$ be a numerical polynomial of degree $d$. Then $f(t))$ can
be represented in the form

\begin{equation}
f(t) =\D\sum_{i=0}^{d}a_{i}{{t+i}\choose i}
\end{equation}
with uniquely defined integer coefficients $a_{i}$.
\end{theorem}

In what follows (until the end of the section), we deal with subsets of
$\mathbb{Z}^{m}$ ($m$ is a positive integer). If $a = (a_{1},\dots, a_{m})\in \mathbb{Z}^{m}$,
then the number $\sum_{i=1}^{m}a_{i}$ will be called the {\em order} of
the $m$-tuple $a$; it is denoted by $\ord a$. Furthermore, the
set $\mathbb{Z}^{m}$ will be considered as the union
\begin{equation}
\mathbb{Z}^{n} = \bigcup_{1\leq j\leq 2^{m}}\mathbb{Z}_{j}^{(m)}
\end{equation}
where $\mathbb{Z}_{1}^{(m)}, \dots, \mathbb{Z}_{2^{m}}^{(m)}$ are
all distinct Cartesian products of $m$ sets each of which is either
$\mathbb{N}$ or ${\bf Z_{-}}=\{a\in \mathbb{Z}|a\leq 0\}$. We assume
that $\mathbb{Z}_{1}^{(m)} = \mathbb{N}^{m}$ and call
$\mathbb{Z}_{j}^{(m)}$ the {\em $j$th orthant} of the set
$\mathbb{Z}^{m}$ ($1\leq j\leq 2^{m}$).

The set $\mathbb{Z}^{m}$ will be considered as a partially ordered set with
the order $\unlhd$ such that $(e_{1},\dots, e_{m})\unlhd (e'_{1},\dots, e'_{m})$ if
and only if $(e_{1},\dots, e_{m})$ and $(e'_{1},\dots, e'_{m})$ belong to the
same orthant $\mathbb{Z}_{k}^{(m)}$ and the $m$-tuple $(|e_{1}|,\dots, |e_{m}|)$
is less than $(|e'_{1}|,\dots, |e'_{m}|)$ with respect to the product order on $\mathbb{N}^{m}$.

In what follows, for any set $A\subseteq \mathbb{Z}^{m}$, $W_{A}$
will denote the set of all elements of $\mathbb{Z}^{m}$ that do not
exceed any element of $A$ with respect to the order $\unlhd$. (Thus,
$w\in W_{A}$ if and only if there is no $a\in A$ such that $a\unlhd
w$.) Furthermore, for any $r\in \mathbb{N}$, $A(r)$ will denote the
set of all elements $x = (x_{1},\dots, x_{m})\in A$ such that $\ord
x\leq r$.

\smallskip

The above notation can be naturally applied to subsets of
$\mathbb{N}^{m}$ (treated as subsets of $\mathbb{Z}^{m}$). If
$E\subseteq \mathbb{N}^{m}$ and $s\in \mathbb{N}$, then $E(s)$ will
denote the set of all $m$-tuples $e = (e_{1},\dots, e_{m})\in E$
such that $\ord e\leq s$. Furthermore, we shall associate with a set
$E\subseteq \mathbb{N}^{m}$ the set $V_{E}\subseteq \mathbb{N}^{m}$
that consists of all $m$-tuples $v = (v_{1},\dots , v_{m})\in
\mathbb{N}^{m}$ that are not greater than or equal to any $m$-tuple
from $E$ with respect to the product order on $\mathbb{N}^{m}$.
(Clearly, an element $v=(v_{1}, \dots , v_{m})\in \mathbb{N}^{m}$
belongs to $V_{E}$ if and only if for any element  $(e_{1},\dots ,
e_{m})\in E$, there exists $i\in \mathbb{N}, 1\leq i\leq m$, such
that $e_{i} > v_{i}$.)

The following two theorems proved, respectively, in \cite[Chapter 0,
section 17]{Kolchin2} and \cite[Chapter 2]{KLMP} introduce certain
numerical polynomials associated with subsets of $\mathbb{N}^{m}$
and give explicit formulas for the computation of these polynomials.

\begin{theorem}
Let $E$ be a subset of $\mathbb{N}^{m}$. Then there exists a
numerical polynomial $\omega_{E}(t)$ with the following properties:

\smallskip

{\em (i)} \,  $\omega_{E}(r) = \Card V_{E}(r))$ for all sufficiently
large $r\in \mathbb{N}$.

\smallskip

{\em (ii)} \, $\deg\omega_{E}$ does not exceed $m$ and
$\deg\omega_{E} = m$ if and only if $E=\emptyset$. In the last case,
$\omega_{E}(t) = {\D{t+m}\choose m}$.
\end{theorem}

The polynomial $\omega_{E}(t)$ is called the {\em dimension
polynomial} of the set $E\subseteq \mathbb{N}^{m}$ .

\begin{theorem}
Let $E = \{e_{1}, \dots, e_{q}\}$ ($q\geq 1$) be a finite subset of
$\mathbb{N}^{m}$.  Let $e_{i} = (e_{i1}, \dots, e_{im})$ \, ($1\leq
i\leq q$) and for any $l\in \mathbb{N}$, $0\leq l\leq q$, let
$\Theta(l,q)$ denote the set of all $l$-element subsets of the set
$\mathbb{N}_{q} = \{1,\dots, q\}$. Furthermore, let
$\bar{e}_{\emptyset j} = 0$ and for any $\theta\in\Theta(l,q)$,
$\theta\neq \emptyset$, let $\bar{e}_{\theta j} = \max
\{e_{ij}\,|\,i\in \theta\}$, $1\leq j\leq m$. (In other words, if
$\theta = \{i_{1},\dots, i_{l}\}$, then $\bar{e}_{\theta j}$ denotes
the greatest $j$th coordinate of the elements $e_{i_{1}},\dots,
e_{i_{l}}$.) Furthermore, let  $b_{\theta}
=\D\sum_{j=1}^{m}\bar{e}_{\theta j}$. Then
\begin{equation}\label{eq:6} \omega_{E}(t) =
\D\sum_{l=0}^{q}(-1)^{l}\D\sum_{\theta\in\Theta(l,\, q)}{t+ m -
b_{\theta}\choose m}
\end{equation}
\end{theorem}

{\bf Remark.} \, It is clear that if $E\subseteq \mathbb{N}^{n}$ and
$E^{\ast}$ is the set of all minimal elements of the set $E$ with
respect to the product order on $\mathbb{N}^{m}$, then the set
$E^{\ast}$ is finite and $\omega_{E}(t) = \omega_{E^{\ast}}(t)$.
Thus, the last theorem gives an algorithm that allows one to find a
numerical polynomial associated with any subset of $\mathbb{N}^{m}$:
one should first find the set of all minimal points of the subset
and then apply Theorem 2.8.

\medskip

The following two results proved in \cite[Section 2.5]{KLMP}
describe dimension polynomials associated with subsets of
$\mathbb{Z}^{m}$.

\begin{theorem}
Let $A$ be a subset of $\mathbb{Z}^{m}$. Then there exists a
numerical polynomial $\phi_{A}(t)$ such that

\smallskip

{\em (i)}\, $\phi_{A}(r) = \Card W_{A}(r)$ for all sufficiently
large $r\in\mathbb{N}$.

\smallskip

{\em (ii)}\, $\deg\phi_{A}\leq m$ and the polynomial $\phi_{A}(t)$
can be written in the form $\phi_{A}(t) =
\D\sum_{i=0}^{m}a_{i}{{t+i}\choose i}$ where $a_{i}\in\mathbb{Z}$
and $2^{m}|a_{m}$.

\smallskip

{\em (iii)}\, $\phi_{A}(t)=0$ if and only if $(0,\dots,0)\in A$.

\smallskip

{\em (iv)}\, If $A = \emptyset$, then $\phi_{A}(t) =
\D\sum_{i=0}^{m}(-1)^{m-i}2^{i}{m\choose i}{{t+i}\choose i}$.
\end{theorem}

\begin{theorem}
With the notation of Theorem 2.9, let us consider a mapping $\rho:
\mathbb{Z}^{m}\longrightarrow\mathbb{N}^{2m}$ such that
$$\rho((e_{1},\dots, e_{m})) =(\max \{e_{1}, 0\}, \dots, \max \{e_{m}, 0\}, \max \{-e_{1}, 0\}, \dots, \max \{-e_{m}, 0\}).$$
Let $B = \rho(A)\bigcup \{\bar{e}_{1},\dots, \bar{e}_{m}\}$ where
$\bar{e}_{i}$ ($1\leq i\leq m$) is a $2n$-tuple in $\mathbb{N}^{2m}$
whose $i$th and $(m+i)$th coordinates are equal to 1 and all other
coordinates are equal to 0. Then $\phi_{A}(t)= \omega_{B}(t)$ where
$\omega_{B}(t)$ is the dimension polynomial of the set $B\subseteq
\mathbb{N}^{2m}$ (i. e., the dimension polynomial introduced in
Theorem 2.7).
\end{theorem}

The polynomial $\phi_{A}(t)$ is called the {\em dimension
polynomial} of the set $A\subseteq \mathbb{Z}^{m}$. It is easy to
see that Theorems 10 and 8 provide an algorithm for computing such a
polynomial.

\section{Difference Dimension Polynomials}

Let $K$ be a difference field with a basic set $\sigma =
\{\alpha_{1},\dots, \alpha_{m}\}$. As before, $T$ will denote the
free commutative semigroup of all elements of the form $\tau =
\alpha_{1}^{k_{1}}\dots \alpha_{m}^{k_{m}}$ ($k_{i}\in\mathbb{N}$
for $i=1,\dots, m$). We define the {\em order} of such an element as
$\ord\tau = \D\sum_{i=1}^{m}k_{i}$, and set $T(r) = \{\tau\in T\,|\,
\ord\tau\leq r\}$ for every $r\in\mathbb{N}$. If the difference
($\sigma$-) field is inversive, that is, all mappings $\alpha_{i}$
are automorphisms of $K$, then $\Gamma$ still denotes the free
commutative group generated by the set $\sigma$. If $\gamma =
\alpha_{1}^{k_{1}}\dots \alpha_{m}^{k_{m}}\in \Gamma$ ($k_{1},\dots,
k_{m}\in \mathbb{N}$), the {\em order} of $\gamma$ is defined as
$\ord\gamma = \D\sum_{i=1}^{m}|k_{i}|$. Furthermore, for any
$r\in\mathbb{N}$, we set $\Gamma(r) = \{\gamma\in\Gamma\,|\,
\ord\gamma\leq r\}$.

\medskip

The following two theorems proved, respectively, in \cite{Levin0}
(see also \cite[Chapter 4]{Levin1} and
 \cite[Theorem 6.4.8]{KLMP} introduce the concepts and provide some
description of dimension polynomials associated with finitely
generated difference and inversive difference field extensions.

\begin{theorem}
Let $K$ be a difference field with a basic set $\sigma =
\{\alpha_{1},\dots , \alpha_{m}\}$ and let $L = K\langle
\eta_{1},\dots,\eta_{n}\rangle$ be a difference field extension of
$K$ generated by a finite set $\eta = \{\eta_{1}, \dots ,
\eta_{n}\}$. Then there exists a numerical polynomial
$\phi_{\eta}(t)$ such that

\smallskip

\noindent {\em (i)} \, $\phi_{\eta}(r) = \trdeg_{K}K(\{\tau\eta_{j}
| \tau \in T(r), 1\leq j\leq n\})$ for all sufficiently large $r\in
\mathbb{Z}$.

\smallskip

\noindent {\em (ii)}\, $\deg \phi_{\eta}(t) \leq m$ and
$\phi_{\eta}(t)$ can be written as $$\phi_{\eta}(t) =
\sum_{i=0}^{m}a_{i}{t+i\choose i}$$ where $a_{0},\dots, a_{m}\in
\mathbb{Z}$;

\smallskip

\noindent {\em (iii)}\, The degree $d$ of the polynomial
$\phi_{\eta}(t)$ and the coefficients $a_{m}$ and $a_{d}$ do not
depend on the choice of the system of generators $\eta$ (clearly,
$a_{d}\neq a_{m}$ if and only if \,$d < m$, that is,\, $a_{m}=0$).
Moreover, $a_{m}$ is equal to the difference transcendence degree of
$L$ over $K$, i. e., to the maximal number of elements
$\xi_{1},\dots,\xi_{k}\in L$ such that the set $\{\tau(\xi_{i}) |
\tau\in T,\, 1\leq i\leq k\}$ is algebraically independent over $K$
(this characteristic of the $\sigma$-field extension is denoted by
$\sigma$-$\trdeg_{K}L$).
\end{theorem}

The numerical polynomial $\phi_{\eta}(t)$ is called a {\bf
difference (or $\sigma$-) dimension polynomial} of the difference
($\sigma$-) field extension $L/K$.

\begin{theorem}
Let $K$ be an inversive  difference field with a basic set $\sigma =
\{\alpha_{1}, \dots, \alpha_{m}\}$ and $L$ an inversive difference
field extension of $K$ generated by a finite set $\eta = \{\eta_{1},
\dots , \eta_{n}\}$. Then there exists a polynomial
$\psi_{\eta|K}(t)$ in one variable $t$ with rational coefficients
such that

\smallskip

{\em (i)}\, $\psi_{\eta|K}(r) = \trdeg_{K}K(\{\gamma(\eta_{j}) |
\gamma \in \Gamma(r), 1\leq j\leq p\})$ for all sufficiently large
integers $r$.

\smallskip

{\em (ii)}\, $\deg \psi_{\eta|K} \leq n$ and the polynomial
$\psi_{\eta|K}(t)$ can be written as $$\psi_{\eta|K}(t) =
\sum_{i=0}^{m}a_{i}2^{i}{t+i\choose i}$$ where $a_{0},\dots,
a_{m}\in \mathbb{Z}$.

\smallskip

{\em (iii)}\, The degree $d$ of the polynomial $\psi_{\eta|K}$ and
the coefficients $a_{m}$ and $a_{d}$ do not depend on the choice of
the system of $\sigma^{\ast}$-generators $\eta$ of the extension
$L/K$.  Furthermore, $a_{m}$ is equal to the difference
transcendence degree of $L$ over $K$.
\end{theorem}

The polynomial $\psi_{\eta|K}(t)$ is called  the {\bf inversive
difference (or $\sigma^{\ast}$-) dimension polynomial} of the
$\sigma^{\ast}$-field extension $L/K$ associated with the set of
$\sigma^{\ast}$-generators $\eta = \{\eta_{1}, \dots , \eta_{n}\}$.

\smallskip

Let $K$ be a difference (respectively, inversive difference) field
with a basic set $\sigma$ and $R$ an algebra of $\sigma$-
(respectively, $\sigma^{\ast}$-) polynomials in difference
indeterminates $y_{1},\dots, y_{n}$ over $K$. Let $P$ be a prime
inversive difference ideal of $R$ and  $\eta = (\eta_{1},\dots,
\eta_{n})$ a generic zero of $P$. Then the dimension polynomial
$\phi_{\eta|K}(t)$ (respectively, $\psi_{\eta|K}(t)$) associated
with the $\sigma$- (respectively, $\sigma^{\ast}$-) field extension
$K\langle \eta_{1},\dots, \eta_{n}\rangle/K$ (respectively,
$K\langle \eta_{1},\dots, \eta_{n}\rangle^{\ast}/K$) is called the
$\sigma$- (respectively, $\sigma^{\ast}$-) {\em dimension
polynomial} of the ideal $P$. It is denoted by $\phi_{P}(t)$
(respectively, $\psi_{P}(t)).$

With the above notation, the inversive difference (or
$\sigma^{\ast}$-) dimension polynomial of a prime system of
algebraic difference equations $g_{i}(y_{1},\dots, y_{n}) = 0$
($i\in I$) is defined as the $\sigma^{\ast}$-dimension polynomial of
the prime $\sigma^{\ast}$-ideal $P$ generated by the
$\sigma^{\ast}$-polynomials $g_{i}$ in $K\{y_{1},\dots,
y_{n}\}^{\ast}$. This dimension polynomial has an interesting
interpretation as a measure of strength in the sense of A. Einstein.
Considering a system of equations in finite differences over a field
of functions in several real variables, one can use A. Einstein's
approach to define the concept of strength of such a system as
follows.  Let
\begin{equation}A_{i}(f_{1},\dots, f_{n}) = 0\end{equation} be a system of equations
in finite differences with respect to $n$ unknown grid functions
$f_{1},\dots, f_{n}$ in $m$ real variables $x_{1},\dots, x_{m}$ with
coefficients in some functional field $K$ over $\mathbb{R}$. We also
assume that the difference grid, whose nodes form the domain of
considered functions, has equal cells of dimension
$h_{1}\times\dots\times h_{m}$ \,($h_{1},\dots, h_{m}\in\mathbb{R}$)
and fills the whole space $\mathbb{R}^{m}$. As an example, one can
consider a field $K$ consisting of the zero function and fractions
of the form $u/v$ where $u$ and $v$ are grid functions defined
almost everywhere and vanishing at a finite number of nodes. (As
usual, we say that a grid function is defined almost everywhere if
there are only finitely many nodes where it is not defined.)

Let us fix some node $\mathcal{P}$ and say that {\em a node
$\mathcal{Q}$ has order $i$} (with respect to $\mathcal{P}$) if the
shortest path from $\mathcal{P}$ to $\mathcal{Q}$ along the edges of
the grid consists of $i$ steps (by a step we mean a path from a node
of the grid to a neighbor node along the edge between these two
nodes). For example, the orders of the nodes in the two-dimensional
case are as follows (a number near a node shows the order of this
node). \setlength{\unitlength}{1cm}
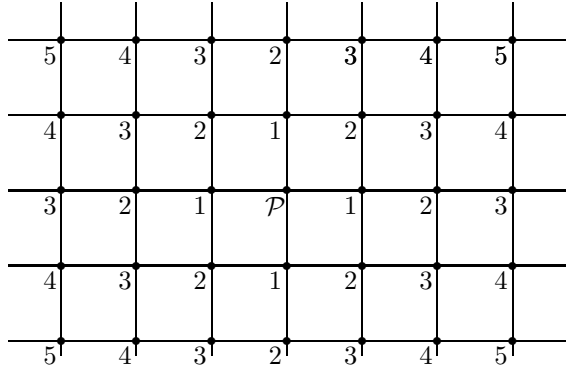
\begin{picture}(9,4.9)
\put(1.6,0.3){\line(1,0){0.7}} \put (2.3,0.30){\circle*{0.1}}
\put(2.3,0.3){\line(1,0){1}} \put (3.3,0.30){\circle*{0.1}} \put
(4.3,0.30){\circle*{0.1}} \put (6.3,0.30){\circle*{0.1}} \put
(5.3,0.30){\circle*{0.1}} \put(3.3,0.3){\line(1,0){1}}
\put(4.3,0.3){\line(1,0){1}} \put(5.3,0.3){\line(1,0){1}}
\put(6.3,0.3){\line(1,0){1}} \put (7.3,0.30){\circle*{0.1}}
\put(7.3,0.3){\line(1,0){1}} \put (8.3,0.30){\circle*{0.1}}
\put(8.3,0.3){\line(1,0){0.7}} \put(2.3,0.3){\line(0,-1){0.2}}
\put(3.3,0.3){\line(0,-1){0.2}} \put(4.3,0.3){\line(0,-1){0.2}}
\put(5.3,0.3){\line(0,-1){0.2}} \put(6.3,0.3){\line(0,-1){0.2}}
\put(7.3,0.3){\line(0,-1){0.2}} \put(8.3,0.3){\line(0,-1){0.2}}
\put(2.3,0.3){\line(0,1){1}} \put(3.3,0.3){\line(0,1){1}}
\put(4.3,0.3){\line(0,1){1}} \put(5.3,0.3){\line(0,1){1}}
\put(6.3,0.3){\line(0,1){1}} \put(7.3,0.3){\line(0,1){1}}
\put(8.3,0.3){\line(0,1){1}} \put(7.3,1.30){\circle*{0.1}}
\put(2.3,1.30){\circle*{0.1}} \put(3.3,1.30){\circle*{0.1}}
\put(4.3,1.30){\circle*{0.1}} \put(5.3,1.30){\circle*{0.1}}
\put(6.3,1.30){\circle*{0.1}} \put(8.3,1.30){\circle*{0.1}}
\put(1.6,1.3){\line(1,0){0.7}} \put(2.3,1.3){\line(1,0){1}}
\put(3.3,1.3){\line(1,0){1}} \put(4.3,1.3){\line(1,0){1}}
\put(5.3,1.3){\line(1,0){1}} \put(6.3,1.3){\line(1,0){1}}
\put(7.3,1.3){\line(1,0){1}} \put(8.3,1.3){\line(1,0){0.7}}

\put(2.3,2.3){\circle*{0.1}} \put(3.3,2.3){\circle*{0.1}}
\put(4.3,2.3){\circle*{0.1}} \put(5.3,2.3){\circle*{0.1}}
\put(6.3,2.3){\circle*{0.1}} \put(7.3,2.3){\circle*{0.1}}
\put(8.3,2.3){\circle*{0.1}}

\put(2.3,3.3){\circle*{0.1}} \put(3.3,3.3){\circle*{0.1}}
\put(4.3,3.3){\circle*{0.1}} \put(5.3,3.3){\circle*{0.1}}
\put(6.3,3.3){\circle*{0.1}} \put(7.3,3.3){\circle*{0.1}}
\put(8.3,3.3){\circle*{0.1}}

\put(2.3,4.3){\circle*{0.1}} \put(3.3,4.3){\circle*{0.1}}
\put(4.3,4.3){\circle*{0.1}} \put(5.3,4.3){\circle*{0.1}}
\put(6.3,4.3){\circle*{0.1}} \put(7.3,4.3){\circle*{0.1}}
\put(8.3,4.3){\circle*{0.1}}

\put(1.6,2.3){\line(1,0){0.7}} \put(2.3,2.3){\line(1,0){1}}
\put(3.3,2.3){\line(1,0){1}} \put(4.3,2.3){\line(1,0){1}}
\put(5.3,2.3){\line(1,0){1}} \put(6.3,2.3){\line(1,0){1}}
\put(7.3,2.3){\line(1,0){1}} \put(8.3,2.3){\line(1,0){0.7}}

\put(1.6,3.3){\line(1,0){0.7}} \put(2.3,3.3){\line(1,0){1}}
\put(3.3,3.3){\line(1,0){1}} \put(4.3,3.3){\line(1,0){1}}
\put(5.3,3.3){\line(1,0){1}} \put(6.3,3.3){\line(1,0){1}}
\put(7.3,3.3){\line(1,0){1}} \put(8.3,3.3){\line(1,0){0.7}}

\put(1.6,4.3){\line(1,0){0.7}} \put(2.3,4.3){\line(1,0){1}}
\put(3.3,4.3){\line(1,0){1}} \put(4.3,4.3){\line(1,0){1}}
\put(5.3,4.3){\line(1,0){1}} \put(6.3,4.3){\line(1,0){1}}
\put(7.3,4.3){\line(1,0){1}} \put(8.3,4.3){\line(1,0){0.7}}

\put(2.3,1.3){\line(0,1){1}} \put(3.3,1.3){\line(0,1){1}}
\put(4.3,1.3){\line(0,1){1}} \put(5.3,1.3){\line(0,1){1}}
\put(6.3,1.3){\line(0,1){1}} \put(7.3,1.3){\line(0,1){1}}
\put(8.3,1.3){\line(0,1){1}}

\put(2.3,2.3){\line(0,1){1}} \put(3.3,2.3){\line(0,1){1}}
\put(4.3,2.3){\line(0,1){1}} \put(5.3,2.3){\line(0,1){1}}
\put(6.3,2.3){\line(0,1){1}} \put(7.3,2.3){\line(0,1){1}}
\put(8.3,2.3){\line(0,1){1}}

\put(2.3,3.3){\line(0,1){1}} \put(3.3,3.3){\line(0,1){1}}
\put(4.3,3.3){\line(0,1){1}} \put(5.3,3.3){\line(0,1){1}}
\put(6.3,3.3){\line(0,1){1}} \put(7.3,3.3){\line(0,1){1}}
\put(8.3,3.3){\line(0,1){1}}

\put(2.3,4.3){\line(0,1){0.5}} \put(3.3,4.3){\line(0,1){0.5}}
\put(4.3,4.3){\line(0,1){0.5}} \put(5.3,4.3){\line(0,1){0.5}}
\put(6.3,4.3){\line(0,1){0.5}} \put(7.3,4.3){\line(0,1){0.5}}
\put(8.3,4.3){\line(0,1){0.5}}

\put(5.15,2.1){\makebox(0,0){$\mathcal{P}$}}
\put(4.15,2.1){\makebox(0,0){$1$}}
\put(6.15,2.1){\makebox(0,0){$1$}}
\put(3.15,2.1){\makebox(0,0){$2$}}
\put(7.15,2.1){\makebox(0,0){$2$}}
\put(2.15,2.1){\makebox(0,0){$3$}}
\put(8.15,2.1){\makebox(0,0){$3$}}

\put(5.15,1.1){\makebox(0,0){$1$}}
\put(5.15,3.1){\makebox(0,0){$1$}}
\put(4.15,1.1){\makebox(0,0){$2$}}
\put(6.15,1.1){\makebox(0,0){$2$}}
\put(3.15,1.1){\makebox(0,0){$3$}}
\put(7.15,1.1){\makebox(0,0){$3$}}
\put(2.15,1.1){\makebox(0,0){$4$}}
\put(8.15,1.1){\makebox(0,0){$4$}}

\put(5.15,4.1){\makebox(0,0){$2$}}
\put(5.15,0.1){\makebox(0,0){$2$}}
\put(4.15,0.1){\makebox(0,0){$3$}}
\put(6.15,4.1){\makebox(0,0){$3$}}
\put(3.15,0.1){\makebox(0,0){$4$}}
\put(7.15,4.1){\makebox(0,0){$4$}}
\put(2.15,0.1){\makebox(0,0){$5$}}
\put(8.15,4.1){\makebox(0,0){$5$}}

\put(6.15,0.1){\makebox(0,0){$3$}}
\put(7.15,0.1){\makebox(0,0){$4$}}
\put(8.15,0.1){\makebox(0,0){$5$}}

\put(6.15,4.1){\makebox(0,0){$3$}}
\put(7.15,4.1){\makebox(0,0){$4$}}
\put(8.15,4.1){\makebox(0,0){$5$}}

\put(6.15,3.1){\makebox(0,0){$2$}}
\put(7.15,3.1){\makebox(0,0){$3$}}
\put(8.15,3.1){\makebox(0,0){$4$}}

\put(4.15,4.1){\makebox(0,0){$3$}}
\put(3.15,4.1){\makebox(0,0){$4$}}
\put(2.15,4.1){\makebox(0,0){$5$}}

\put(4.15,3.1){\makebox(0,0){$2$}}
\put(3.15,3.1){\makebox(0,0){$3$}}
\put(2.15,3.1){\makebox(0,0){$4$}}
\end{picture}\\
Let us consider the values of the unknown grid functions
$f_{1},\dots, f_{n}$ at the nodes whose order does not exceed $r$\,
($r\in\mathbb{N}$). If $f_{1},\dots, f_{n}$ should not satisfy any
system of equations (or any other condition), their values at nodes
of any order can be chosen arbitrarily. Because of the system in
finite differences (and equations obtained from the equations of the
system by transformations of the form $f_{j}(x_{1},\dots,
x_{m})\mapsto f_{j}(x_{1}+k_{1}h_{1},\dots, x_{m}+k_{m}h_{m})$ with
$k_{1},\dots, k_{m}\in \mathbb{Z}$, $1\leq j\leq n$), the number of
independent values of the functions $f_{1},\dots, f_{n}$ at the
nodes of order $\leq r$ decreases. This number, which is a function
of $r$, is considered as the ''measure of strength'' of the system
in finite differences (in the sense of A. Einstein). We denote it by
$S_{r}$.

With the above conventions, suppose that the transformations
$\alpha_{j}$ of the field of coefficients $K$ defined by
$$\alpha_{j}f(x_{1},\dots, x_{m}) = f(x_{1},\dots,x_{j-1}, x_{j}+h_{j},\dots,
x_{m})\hspace{0.2in} (1\leq j\leq m)$$ are automorphisms of this
field. Then $K$ can be considered as an inversive difference field
with the basic set $\sigma = \{\alpha_{1},\dots, \alpha_{m}\}$.
Furthermore, assume that the replacement of the unknown functions
$f_{i}$ by $\sigma^{\ast}$-indeterminates $y_{i}$ ($i=1,\dots, n$)
in the ring $K\{y_{1},\dots, y_{n}\}^{\ast}$ leads to a prime system
of algebraic $\sigma^{\ast}$-equations (then the original system of
equations in finite differences is also called {\em prime}). The
$\sigma^{\ast}$-dimension polynomial $\psi(t)$ of the latter system
is said to be the $\sigma^{\ast}$-dimension polynomial of the given
system in finite differences.

Clearly, $\psi(r) = S_{r}$ for any $r\in \mathbb{N}$, so the
$\sigma^{\ast}$-dimension polynomial of a prime system of equations
in finite differences is the measure of strength of such a system in
the sense of A. Einstein.

\medskip

The main method of computation of difference (respectively,
inversive difference) dimension polynomials is based on the
construction of characteristic sets in the algebra difference
(respectively, inversive difference) polynomials. In what follows we
will consider the case of inversive difference ($\sigma^{\ast}$-)
polynomials, they reflect finite-difference approximations of PDEs
where shifts of arguments can be of any sign.

Let $K$ be an inversive difference field with a basic set of
automorphisms (translations) $\sigma = \{\alpha_{1},\dots,
\alpha_{m}\}$ and $\Gamma = \Gamma_{\sigma}$ the free commutative
group generated by $\sigma$. Let $\bar{\mathbb{Z}}_{-}$ denote the
set of all non-positive integers and let $\mathbb{Z}^{(m)}_{1},
\mathbb{Z}^{(m)}_{2}, \dots, \mathbb{Z}^{(m)}_{2^{m}}$ be all
orthants of $\mathbb{Z}^{m}$. Recall (see section 2) that they are
distinct Cartesian products of $m$ factors each of which is either
$\mathbb{N}$ or $\bar{\mathbb{Z}}_{-}$ (we assume that
$\mathbb{Z}^{(m)}_{1}=\mathbb{N}$). For any $j=1,\dots, 2^{m}$, we
set $\Gamma_{j} = \{\gamma = \alpha_{1}^{k_{1}}\dots
\alpha_{m}^{k_{m}}\in \Gamma\,|\,(k_{1},\dots,
k_{m})\in\mathbb{Z}^{(m)}_{j}\}.$ As before, for any element $\gamma
= \alpha_{1}^{k_{1}}\dots \alpha_{m}^{k_{m}}\in \Gamma$, the number
$\ord\,\gamma = \D\sum_{i=1}^{m}|k_{i}|$ will be called the {\em
order} of $\gamma$.

Let $K\{y_{1},\dots, y_{n}\}^{\ast}$ be the ring of
$\sigma^{\ast}$-polynomials in $\sigma^{\ast}$-indeterminates
$y_{1},\dots, y_{n}$ over $K$ and let $\Gamma Y$ denote the set
$\{\gamma y_{i} | \gamma \in \Gamma, 1\leq i\leq m \}$ whose
elements are called {\em terms} (here and below we often write
$\gamma y_{i}$ for $\gamma(y_{i}$)). By the order of a term $u =
\gamma y_{j}$ we mean the order of the element $\gamma \in \Gamma$.
Setting $(\Gamma Y)_{j} = \{\gamma y_{i} | \gamma \in \Gamma_{j},
1\leq i\leq n \}$ ($j=1,\dots, 2^{m}$) we obtain a representation of
the set of terms as a union $$\Gamma Y =
\D\bigcup_{j=1}^{2^{n}}(\Gamma Y)_{j}.$$

\begin{definition} A term $v\in \Gamma Y$ is called a transform
of a term $u\in \Gamma Y$ if and only if $u$ and $v$ belong to the
same set $(\Gamma Y)_{j} \, (1\leq j\leq 2^{m})$ and $v = \gamma u$
for some $\gamma \in \Gamma_{j}$. If $\gamma \neq 1$, $v$ is said to
be a proper transform of $u$.
\end{definition}

\begin{definition} A well-ordering of the set of terms
$\Gamma Y$ is called a ranking of the family of
$\sigma^{\ast}$-indeterminates $y_{1},\dots, y_{n}$ (or a ranking of
the set $\Gamma Y$) if it satisfies the following conditions. {\em
(}We use the standard symbol $\leq$ for the ranking; it will be
always clear what order is denoted by this symbol.{\em )}

\smallskip

{\em (i)}\, If $u\in (\Gamma Y)_{j}$ and $\gamma \in \Gamma_{j}$
($1\leq j\leq 2^{m}$), then $u\leq \gamma u$.

\smallskip

{\em (ii)}\, If $u, v\in (\Gamma Y)_{j}$ ($1\leq j\leq 2^{m}$),
$u\leq v$ and $\gamma \in \Gamma_{j}$, then $\gamma u \leq \gamma
v$.

A ranking of the $\sigma^{\ast}$-indeterminates $y_{1},\dots, y_{m}$
is called orderly if for any $j=1,\dots, 2^{m}$ and for any two
terms $u, v \in (\Gamma Y)_{j}$, the inequality $\ord\,u < \ord\,v$
implies that $u < v$ (as usual, $v < w$ means $v\leq w$ and $v\neq
w$).
\end{definition}

As an example of an orderly ranking of the
$\sigma^{\ast}$-indeterminates $y_{1},\dots, y_{s}$ one can consider
the {\em standard ranking} defined as follows: $u =
\alpha_{1}^{k_{1}}\dots \alpha_{m}^{k_{m}}y_{i}\leq v =
\alpha_{1}^{l_{1}}\dots \alpha_{m}^{l_{m}}y_{j}$ if and only if the
$(2m+2)$-tuple $(\D\sum_{\nu = 1}^{m}|k_{\nu}|, |k_{1}|,\dots,
|k_{m}|, k_{1},\dots, k_{m}, i)$ is less than or equal to the
$(2m+2)$-tuple $(\D\sum_{\nu = 1}^{m}|l_{\nu}|, |l_{1}|,\dots,
|l_{m}|, l_{1},\dots, l_{m}, j)$ with respect to the lexicographic
order on $\mathbb{Z}^{2m+2}$.

\smallskip

In what follows, we assume that an orderly ranking $\leq$ of the set
of $\sigma^{\ast}$-indeterminates $y_{1},\dots, y_{n}$ is fixed. If
$A\in K\{y_{1},\dots, y_{n}\}^{\ast}$, then the greatest (with
respect to the ranking $\leq$) term that appears in $A$ is called
the {\em leader} of $A$; it is denoted by $u_{A}$. If $u = u_{A}$
and $d=\deg_{u}A$, then the $\sigma^{\ast}$-polynomial $A$ can be
written as $A = I_{d}u^{d} + I_{d-1}u^{d-1}+\dots + I_{0}$ where
$I_{k} (0\leq k\leq d)$ do not contain $u$. The
$\sigma^{\ast}$-polynomial $I_{d}$ is called the {\em initial} of
$A$; it is denoted by $I_{A}$.

\begin{definition}
Let $A, B\in K\{y_{1}\dots, y_{n}\}$. We say that $A$ has higher
rank than $B$ and write $\rk A > \rk B$ if either $A\notin K,\, B\in
K$, or $u_{A}$ has higher rank than $u_{B}$, or $u_{A} = u_{B}$ and
$\deg_{u_{A}}A > \deg_{u_{A}}B$. If $u_{A} = u_{B}$ and
$\deg_{u_{A}}A = \deg_{u_{A}}B$, we say that $A$ and $B$ have the
same rank and write $\rk A = \rk B$.
\end{definition}

Note that distinct $\sigma^{\ast}$-polynomials can have the same
rank and if $A\notin K$, then $I_{A}$ has lower rank than $A$.

\begin{definition}
Let $A, B\in K\{y_{1},\dots, y_{n}\}^{\ast}$. The
$\sigma^{\ast}$-polynomial $A$ is said to be {\em reduced} with
respect to $B$ if $A$ does not contain any power of a transform
$\gamma u_{B}$ ($\gamma \in \Gamma$) whose exponent is greater than
or equal to $\deg_{u_{B}}B$. If $\mathcal{A}\subseteq
K\{y_{1},\dots, y_{n}\}\setminus K$, then a
$\sigma^{\ast}$-polynomial $A\in K\{y_{1},\dots, y_{n}\}$, is said
to be reduced with respect to $\mathcal{A}$ if $A$ is reduced with
respect to every element of the set $\mathcal{A}$.

\smallskip

A set $\mathcal{A}\subseteq K\{y_{1},\dots, y_{n}\}^{\ast}$ is said
to be autoreduced if either it is empty or $\mathcal{A}\bigcap K =
\emptyset$ and every element of $\mathcal{A}$ is reduced with
respect to all other elements of the set $\mathcal{A}$.
\end{definition}

The proof of the following proposition can be obtained by mimicking
the proof of the corresponding statement about autoreduced sets of
differential polynomials, see \cite[Chapter 2]{Kolchin2}.

\begin{proposition}
Every autoreduced set is finite and distinct elements of an
autoreduced set have distinct leaders.
\end{proposition}
\begin{theorem}{\em (\cite[Theorem 2.4.7]{Levin1})} Let $\mathcal{A} =
\{A_{1},\dots, A_{p}\}$ be an autoreduced subset in the ring of
$\sigma^{\ast}$-polynomials  $K\{y_{1},\dots, y_{n}\}^{\ast}$ and
let $D\in K\{y_{1},\dots, y_{n}\}^{\ast}$. Furthermore, let
$I(\mathcal{A})$ denote the set of all $\sigma^{\ast}$-polynomials
$B\in K\{y_{1},\dots, y_{n}\}$ such that either $B =1$ or $B$ is a
product of finitely many polynomials of the form $\gamma(I_{A_{i}})$
where $\gamma \in \Gamma,\, i=1,\dots, p$. Then there exist
$\sigma^{\ast}$-polynomials $J\in I(\mathcal{A})$ and $D_{0}\in
K\{y_{1},\dots, y_{n}\}$ such that $D_{0}$ is reduced with respect
to $\mathcal{A}$ and $JD\equiv D_{0} (mod\, [\mathcal{A}])$.
\end{theorem}

Note that, with the notation of the last theorem, the process of
reduction that leads to the $\sigma^{\ast}$-polynomials $J\in
I(\mathcal{A})$ and $D_{0}$ is algorithmic; the steps of the
corresponding algorithm can be obtained by mimicking the steps in
the proof of Theorem 2.4.1 of \cite{Levin1}. The
$\sigma^{\ast}$-polynomial $D_{0}$ is called the {\em remainder} of
$D$ with respect to $\mathcal{A}$. We also say that $D$ {\em reduces
to $D_{0}$ modulo $\mathcal{A}$}.

In what follows elements of an autoreduced set in $K\{y_{1},\dots,
y_{n}\}^{\ast}$ will be always written in the order of increasing
rank. With this assumption we introduce the following partial order
on the set of all autoreduced sets.

\begin{definition}
Let $\mathcal{A} = \{A_{1},\dots, A_{p}\}$ and $\mathcal{B} =
\{B_{1},\dots, B_{q}\}$ be two autoreduced sets of
$\sigma^{\ast}$-polynomials in $K\{y_{1},\dots, y_{n}\}^{\ast}$. We
say that $\mathcal{A}$ has lower rank than $\mathcal{B}$ and write
$\rk\mathcal{A} < \rk\mathcal{B}$ if either there exists
$k\in\mathbb{N},\, 1\leq k\leq \min\{p, q\}$, such that $\rk A_{i} =
\rk B_{i}$ for $i=1,\dots, k-1$ and $\rk A_{k} < \rk B_{k}$, or $p >
q$ and $\rk A_{i} = \rk B_{i}$ for $i=1,\dots, q$.
\end{definition}

Mimicking the arguments of \cite[Chapter 1, Section 9]{Kolchin2},
one obtains that every nonempty family of autoreduced subsets of
$K\{y_{1},\dots, y_{n}\}^{\ast}$ contains an autoreduced set of
lowest rank. In particular, if $\emptyset\neq J\subseteq
F\{y_{1},\dots, y_{n}\}^{\ast}$, then the set $J$ contains an
autoreduced set of lowest rank called a {\bf characteristic set} of
$J$.

\begin{proposition}{\em (\cite[Proposition 2.4.8]{Levin1})}
Let $K$ be an inversive difference field with a basic set $\sigma$,
$J$ a $\sigma^{\ast}$-ideal of the algebra of $\sigma$-polynomials
$K\{y_{1},\dots, y_{n}\}^{\ast}$, and $\mathcal{A}$ a characteristic
set of $J$. Then

\smallskip

{\em (i)}\, The ideal $J$ does not contain nonzero
$\sigma^{\ast}$-polynomials reduced with respect to $\mathcal{A}$.
In particular, if $A\in\mathcal{A}$, then $I_{A}\notin J$.

\smallskip

{\em (ii)}\, If $J$ is a prime $\sigma^{\ast}$-ideal, then $J =
[\mathcal{A}]^{\ast}:\Upsilon(\mathcal{A})$ where
$\Upsilon(\mathcal{A})$ denotes the set of all finite products of
elements of the form $\gamma(I_{A})$ ($\gamma \in \Gamma_{\sigma},
A\in\mathcal{A}$).
\end{proposition}

If $K$ is an inversive difference ($\sigma^{\ast}$-) field, then a
$\sigma^{\ast}$-ideal of the ring of $\sigma^{\ast}$-polynomials
$K\{y_{1},\dots, y_{n}\}^{\ast}$ is called {\em linear} if it is
generated (as a $\sigma^{\ast}$-ideal) by homogeneous linear
$\sigma^{\ast}$-polynomials, i. e., $\sigma^{\ast}$-polynomials of
the form $\D\sum_{i=1}^{p}a_{i}\gamma_{i}y_{k_{i}}$ ($a_{i}\in K,
\gamma_{i}\in \Gamma, 1\leq k_{i}\leq n$ for $i=1,\dots, p$). As it
is shown in \cite[Proposition 2.4.9]{Levin1}, every linear
$\sigma^{\ast}$-ideal in $K\{y_{1},\dots, y_{n}\}^{\ast}$ is prime.

\begin{definition} Let $K$ be an inversive difference ($\sigma^{\ast}$-)
field and $\mathcal{A}$ an autoreduced set in $K\{y_{1},\dots,
y_{n}\}^{\ast}$ that consists of linear $\sigma^{\ast}$-polynomials.
The set $\mathcal{A} $ is called {\bf coherent} if the following two
conditions hold.

\smallskip

(i)\, If $A\in\mathcal{A}$ and $\gamma \in \Gamma$, then $\gamma A$
reduces to zero modulo $\mathcal{A}$.

\smallskip

(ii)\, If $A, B\in\mathcal{A}$ and $v = \gamma_{1}u_{A} =
\gamma_{2}u_{B}$ is a common transform of the leaders $u_{A}$ and
$u_{B}$ ($\gamma_{1}, \gamma_{2}\in\Gamma$), then the
$\sigma^{\ast}$-polynomial $(\gamma_{2}I_{B})(\gamma_{1}A) -
(\gamma_{1}I_{A})(\gamma_{2}B)$ reduces to zero modulo
$\mathcal{A}$.
\end{definition}

\begin{theorem}{\em (\cite[Theorem 6.5.3]{KLMP})}
Let $K$ be an inversive difference ($\sigma^{\ast}$-) field and $J$
a linear $\sigma^{\ast}$-ideal of $K\{y_{1},\dots, y_{n}\}^{\ast}$.
Then any characteristic set of $J$ is a coherent autoreduced set of
linear $\sigma^{\ast}$-polynomials. Conversely, if
$\mathcal{A}\subseteq K\{y_{1},\dots, y_{n}\}^{\ast}$ is any
coherent autoreduced set consisting of linear
$\sigma^{\ast}$-polynomials, then $\mathcal{A}$ is a characteristic
set of the linear $\sigma^{\ast}$-ideal $[\mathcal{A}]^{\ast}$.
\end{theorem}

\begin{remark}
Analyzing the proof of the last theorem given in \cite{KLMP}, one
can see that this proof also works for the case when all elements of
a coherent autoreduced set are {\bf quasi-linear
$\sigma^{\ast}$-polynomials}, i. e., $\sigma^{\ast}$-polynomials
that are linear with respect to their leaders. (Of course, in this
case the $\sigma^{\ast}$-ideal $[\mathcal{A}]^{\ast}$ is not linear;
moreover, it is not necessarily prime.)
\end{remark}

\begin{proposition}
With the above notation, let $A$ be a quasi-linear $\sigma$-polynomial in the ring of difference ($\sigma$-) polynomials
$K\{y_{1},\dots, y_{n}\}$. Then

{\em (i)}\, If
\begin{equation}
M = \sum_{i=1}^{p}C_{i}\tau_{i}A
\end{equation}
where $\tau_{i}\in T$ and $C_{i}\in K\{y_{1},\dots, y_{n}\}$ ($1\leq i\leq p$),
then $M$ contains the leader of some $\tau_{i}A$.

{\em (ii)}\, The $\sigma$-ideal $[A]$ of the ring $K\{y_{1},\dots,
y_{n}\}$ is prime.

Furthermore, similar statements hold when the $\sigma$-field $K$ is inversive
and $A$ is a quasi-linear $\sigma^{\ast}$-polynomial in $K\{y_{1},\dots, y_{n}\}^{\ast}$.
(In this case, $\tau_{i}\in\Gamma$ and $\tau_{i}u_{A}$ are transforms of the leader of $A$ in the
sense of Definition 3.3.)
\end{proposition}

\begin{proof}
We will prove our statement in the case when $A$ is a quasi-linear $\sigma$-polynomial in
$K\{y_{1},\dots, y_{n}\}$ (the proof for a $\sigma^{\ast}$-polynomial in
$K\{y_{1},\dots, y_{n}\}^{\ast}$ is similar).

Suppose that statement (i) is not true. Let $p$ be the smallest
positive integer such that one has equality (3.2) where $M$ does not
contain any $u_{\tau_{i}A} = \tau_{i}u_{A}$. Without loss of
generality we can assume that  $\tau_{1} < \dots <\tau_{p}$ (so that
$u_{\tau_{1}A} <\dots < u_{\tau_{p}A}$) and $\tau_{p}A$ does not
divide any $C_{i}$ for $i=1,\dots, p-1$ (otherwise, representation
(3.2) is not minimal in the above sense). If one writes $C_{i}$
($1\leq i\leq p-1$) as a polynomial of $u_{\tau_{p}A}$, $C_{i} =
\sum_{i=0}^{d_{i}}I_{i}u_{\tau_{p}A}^{i}$, where the
$\sigma$-polynomials $I_{i}$ do not contain $u_{\tau_{p}A}$, then
$\deg_{u_{\tau_{p}A}}(C_{i} -
I_{d_{i}}(u_{\tau_{p}A})^{d_{i}-1}\tau_{p}A < d_{i}$, so we can
obtain a $\sigma^{\ast}$-polynomial $D_{i}$ such that $D_{i}$ does
not contain $u_{\tau_{p}A}$ and $D_{i}\equiv C_{i}\, (mod
(\tau_{p}A)\,)$. Then $\tau_{p}A$ divides the $\sigma$-polynomial
$M' = M-\sum_{i=1}^{p-1}D_{i}(\tau_{i}A)$. However, $M'$ does not
contain $u_{\tau_{p}A}$ and therefore cannot be divisible by
$\tau_{p}A$. It follows that $M'=0$, hence $M =
\sum_{i=1}^{p-1}D_{i}(\tau_{i}A)$, contrary to the assumption about
the minimality of $p$ in (3.2).

In order to show that the $\sigma$-ideal $[A]$ is prime, consider
two $\sigma$-polynomials $B, C\in K\{y_{1},\dots, y_{n}\}\setminus
[A]$. Let $\tau_{1}u_{A},\dots, \tau_{q}u_{A}$ be all transforms of
$u_{A}$ that appear in $B$ such that $\tau_{1} <\dots < \tau_{q}$.
As above, one can subtract from $B$ an appropriate linear
combination of $\tau_{p}A$ with coefficients in $K\{y_{1},\dots,
y_{n}\}$ in order to eliminate $\tau_{q}u_{A}$ in $B$, so that the
difference will contain only transforms $\tau_{i}u_{A}$ of $u_{A}$
with $i < q$. Repeating this process we arrive at a
$\sigma$-polynomial $B'$ such that $B - B'\in [A]$ and $B'$ does not
contain any $\tau u_{A}$ ($\tau\in T$). Similarly we can find $C'\in
K\{y_{1},\dots, y_{n}\}$ such that $C'$  does not contain any $\tau
u_{A}$ ($\tau\in T$) and $C - C'\in [A]$. Then $B'C'\in [A]$ and
$B'C'$ contains no term of the form $\tau u_{A}$ ($\tau\in T$). This
contradiction with the first part of the proposition shows that
$BC\notin [A]$, so the $\sigma$-ideal $[A]$ is prime.
\end{proof}

The following result is a direct consequence of Theorem 3.12 (taking into account remark 3.13)

\begin{proposition} Let $K$ be an inversive difference
field with a basic set $\sigma$ and let $\preccurlyeq$ be a preorder
on $K\{y_{1},\dots, y_{n}\}^{\ast}$ such that $A_{1}\preccurlyeq
A_{2}$ if and only if $u_{A_{2}}$ is a transform of $u_{A_{1}}$.
Furthermore, let $A$ be a quasi-linear $\sigma^{\ast}$-polynomial in
$K\{y_{1},\dots, y_{n}\}^{\ast}\setminus K$ and $\Gamma A = \{\gamma
A\,|\,\gamma\in\Gamma\}$. Then the set of all minimal (with respect
to $\preccurlyeq$) elements of $\Gamma A$ is a characteristic set of
the $\sigma^{\ast}$-ideal $[A]^{\ast}$.
\end{proposition}

Theorem 3.12 and Proposition 3.15 imply the following method of
constructing a characteristic set of a proper linear
$\sigma^{\ast}$-ideal $I$ in the ring of $\sigma^{\ast}$-polynomials
$K\{y_{1},\dots, y_{s}\}^{\ast}$ (a similar method can be used for
building a characteristic set of a linear $\sigma$-ideal in the ring
of difference polynomials $K\{y_{1},\dots, y_{s}\}$). Suppose that
$I = [B_{1},\dots, B_{p}]^{\ast}$ where $B_{1},\dots, B_{p}$ are
linear $\sigma^{\ast}$-polynomials and $B_{1} < \dots < B_{p}$. It
follows from Theorem 2.4.11 that one should find a coherent
autoreduced set $\Phi \subseteq K\{y_{1},\dots, y_{s}\}^{\ast}$ such
that $[\Phi]^{\ast} = I$. Such a set can be obtained from the set
$\mathcal{B} = \{B_{1},\dots, B_{p}\}$ via the following two-step
procedure.

\smallskip

{\em Step 1.} Constructing an autoreduced set $\mathcal{A}\subseteq
I$ such that $[\mathcal{A}]^{\ast} = I$.

If $\mathcal{B}$ is autoreduced, set $\mathcal{A} = \mathcal{B}$. If
$\mathcal{B}$ is not autoreduced, choose the smallest $i$ ($1\leq
i\leq p$) such that some $\sigma^{\ast}$-polynomial $B_{j}$, $1\leq
i < j\leq p$, is not reduced with respect to $B_{i}$. Replace
$B_{j}$ by its remainder with respect to $B_{i}$ and arrange the
$\sigma^{\ast}$-polynomials of the new set $\mathcal{B}_{1}$ in
ascending order. Then apply the same procedure to the set
$\mathcal{B}_{1}$ and so on. After each iteration the number of
$\sigma^{\ast}$-polynomials in the set does not increase, one of
them is replaced by a $\sigma^{\ast}$-polynomial of lower or equal
rank, and the others do not change. Therefore, the process
terminates after a finite number of steps when we obtain a desired
autoreduced set $\mathcal{A}$.

\smallskip

{\em Step 2.} Constructing a coherent  autoreduced set
$\Phi\subseteq I$.

Let $\mathcal{A}_{0} = \mathcal{A}$ be an autoreduced subset of $I$
such that $[\mathcal{A}]^{\ast} = I$. If $\mathcal{A}$ is not
coherent, we build a new autoreduced set $\mathcal{A}_{1}\subseteq
I$ by adding to $\mathcal{A}_{0}$ new $\sigma^{\ast}$-polynomials of
the following types.

(a) $\sigma^{\ast}$-polynomials
$(\gamma_{1}I_{A_{1}})\gamma_{2}A_{2} -
(\gamma_{2}I_{A_{2}})\gamma_{1}A_{1}$ constructed for every pair
$A_{1}, A_{2}\in \mathcal{A}_{0}$ such that the leaders $u_{A_{1}}$
and $u_{A_{1}}$ have a common transform $v = \gamma_{1}u_{A_{1}} =
\gamma_{2}u_{A_{2}}$ and $(\gamma_{1}I_{A_{1}})\gamma_{2}A_{2} -
(\gamma_{2}I_{A_{2}})\gamma_{1}A_{1}$ is not reducible to zero
modulo $\mathcal{A}_{0}$.

(b) $\sigma^{\ast}$-polynomials of the form $\gamma A$ ($\gamma \in
\Gamma, A\in\mathcal{A}_{0}$) that are not reducible to zero modulo
$\mathcal{A}_{0}$.

It is clear that $\rk\mathcal{A}_{1} < \rk\mathcal{A}_{0}$. Applying
the same procedure to $\mathcal{A}_{1}$ and continuing in the same
way, we obtain autoreduced subsets $\mathcal{A}_{0},
\mathcal{A}_{1},\dots $ of $I$ such that $\rk\mathcal{A}_{i+1} <
\rk\mathcal{A}_{i}$ for $i=0, 1,\dots$. Obviously, the process
terminates after finitely many steps, so we obtain an autoreduced
set $\Phi \subseteq I$ such that $\Phi = \mathcal{A}_{k} =
\mathcal{A}_{k+1}=\dots$ for some $k\in\mathbb{N}$. It is easy to
see that $\Phi$ is coherent, so it is a characteristic set of the
ideal $I$.

\medskip

The following result, whose proof can be extracted from the proof of
Theorem 4.2.5 of \cite{Levin1}, gives a method of computation of the
$\sigma^{\ast}$-dimension polynomial associated with a reflexive
prime difference ideal of the ring of $\sigma^{\ast}$-polynomials
$K\{y_{1},\dots, y_{n}\}^{\ast}$. Therefore, it provides a method of
computation of the Einstein's strength of a prime system of
algebraic partial difference equations. In the next part of the
paper this result will be used for the evaluation of the strength of
systems of difference equations that represent finite-difference
schemes for PDEs describing certain chemical processes.

\begin{theorem}
Let $K$ be an inversive difference field with a basic set of
automorphisms $\sigma = \{\alpha_{1},\dots, \alpha_{m}\}$,
$R=K\{y_{1},\dots, y_{n}\}^{\ast}$ the ring of
$\sigma^{\ast}$-polynomials over $K$, and $P$ a reflexive prime
difference ideal of $R$. Let $L$ denote the quotient field of $R/P$
treated as the $\sigma^{\ast}$-field extension
$K\langle\eta_{1},\dots,\eta_{n}\rangle^{\ast}$ of $K$ where
$\eta_{i}$ ($1\leq i\leq n$) is the canonical image of $y_{i}$ in
$R/P$. Then the $\sigma^{\ast}$-dimension polynomial
$\psi_{\eta|K}(t)$ of $L/K$ associated with the set of
$\sigma^{\ast}$-generators $\eta = \{\eta_{1},\dots,\eta_{n}\}$ (see
Theorem 3.2) can be found as follows.

Let $\mathcal{A} = \{A_{1},\dots, A_{p}\}$ be a characteristic set
of the $\sigma^{\ast}$-ideal $P$ and for every $i=1,\dots, n$, let
$$E_{i} = \{(e_{i1},\dots,
e_{im})\in\mathbb{Z}^{m}\,|\,\alpha_{1}^{e_{i1}}\dots\alpha_{m}^{e_{im}}y_{i}\,\,\,
\text{is the leader of some element of}\,\,\,  \mathcal{A}\}$$ (of
course, some sets $E_{i}$ might be empty). Then
$$\psi_{\eta|K}(t) = \sum_{i=1}^{n}\phi_{E_{i}}(t)$$ where
$\phi_{E_{i}}(t)$ is the dimension polynomial of the set
$E_{i}\subseteq\mathbb{Z}^{m}$ whose existence is established by
Theorem 2.9.
\end{theorem}

\section{Evaluation of the Einstein's strength of difference schemes
for some reaction-diffusion equations}

{\bf 1.}\, {\bf The diffusion equation in one spatial dimension} for a
constant collective diffusion coefficient $a$ and unknown function
$u(x, t)$ describing the density of the diffusing material at given
position $x$ and time $t$ is as follows:
\begin{equation}{\frac{\partial u(x, t)}{\partial t}} =
a{\frac{\partial^{2}u(x, t)}{\partial x^{2}}}.\end{equation} ($a$ is
a constant). Let us compute the strength of difference equations
that arise from three most common difference schemes for equation
(4.1).

\medskip

\large
\begin{center}
Strength of the forward difference scheme
\end{center}
\normalsize
\medskip

In order to obtain the forward difference scheme for the diffusion
equation (4.1), the occurrences of $\D\frac{\partial
u(x,t)}{\partial x}$ and $\D\frac{\partial u(x,t)}{\partial t}$ are
replaced by $u(x + 1, t) - u(x, t)$ and $u(x, t + 1) - u(x, t)$,
respectively (after the appropriate rescaling, one can use these
expressions instead of the standard approximations $\D\frac{u(x + h,
t) - u(x, t)}{h}$ and $\D\frac{u(x, t) - u(x, t+h)}{h}$ with a small
step $h$).

We obtain the equation in finite differences
\begin{equation}
u(x, t + 1) - u(x, t) = a(u(x + 2, t) - 2u(x + 1, t) + u(x, t)).
\end{equation}

Let $K$ be an inversive difference functional field with basic set
$\sigma = \{\alpha_{1}:f(x, t)\mapsto f(x+1, t),\, \alpha_{2}:f(x,
t)\mapsto f(x, t+1)\}$ ($f(x, t)\in K$) containing $a$ and let
$K\{y\}^{\ast}$ be the ring of $\sigma^{\ast}$-polynomials in one
$\sigma^{\ast}$-indeterminate $y$ over $K$. Treating $y$ as the
unknown function $u(x, t)$ in the equation (4.2), we can write this
equation as
\begin{equation}
a\alpha_{1}^{2}y - 2a\alpha_{1}y - \alpha_{2}y + (a+1)y = 0.
\end{equation}
Since the left-hand side of the last equation is a linear
$\sigma^{\ast}$-polynomial, it generates a linear (and therefore a
prime) $\sigma^{\ast}$-ideal $P = [A]^{\ast}$ in $K\{y\}^{\ast}$.

\medskip

Applying Proposition 3.15, we obtain a characteristic set $\mathcal{A}
= \{A_{1}, A_{2}, A_{3}, A_{4}\}$ of the ideal $P$ where
\begin{align*}
A_{1} = A = a\alpha_{1}^{2}y - 2a\alpha_{1}y - \alpha_{2}y +
(a+1)y,\\ A_{2} = \alpha_{1}^{-1}A = -\alpha_{1}^{-1}\alpha_{2}y +
a\alpha_{1} y + (a+1)\alpha_{1}^{-1}y + 2ay,\\ A_{3} =
\alpha_{1}^{-1}\alpha_{2}^{-1}A = a\alpha_{1}\alpha_{2}^{-1}y
+(a+1)\alpha_{1}^{-1}\alpha_{2}^{-1}y - \alpha_{1}^{-1}y -
2a\alpha_{2}^{-1}y,\\A_{4} = \alpha_{1}^{-2}\alpha_{2}^{-1}A =
(a+1)\alpha_{1}^{-2}\alpha_{2}^{-1}y + \alpha_{1}^{-2}y
+2a\alpha_{1}\alpha_{2}^{-1}y + a\alpha_{2}^{-1}y.
\end{align*}

The leaders of these $\sigma^{\ast}$-polynomials are
$\alpha_{1}^{2}y,\, \alpha_{1}^{-1}\alpha_{2}y,\,
\alpha_{1}\alpha_{2}^{-1}y$, and $\alpha_{1}^{-2}\alpha_{2}^{-1}y$,
respectively (they are written first in the
$\sigma^{\ast}$-polynomials $A_{i}$ above). Therefore, the
$\sigma^{\ast}$-dimension polynomial of equation (4.3) is equal to
the dimension polynomial of the subset $$E = \{((2, 0), (-1, 1), (1,
-1), (-2, -1)\}$$ of $\mathbb{Z}^{2}$. Applying the results of
theorems 2.10 and 2.8 we obtain that the $\sigma^{\ast}$-dimension
polynomial of equation (4.3) that expresses the Einstein's strength
of the forward difference scheme for (4.1) is
$$\psi_{Forw}(t) = 5t.$$

\medskip

\large
\begin{center}
Strength of the symmetric difference scheme
\end{center}
\normalsize
\medskip

The symmetric difference scheme for the diffusion equation (4.1) is
obtained by replacing the partial derivatives $\D\frac{\partial^{2}
u(x,t)}{\partial x^{2}}$ and $\D\frac{\partial u(x,t)}{\partial t}$
with $u(x + 1, t) - 2u(x, t)$ and $u(x, t + 1) - u(x, t-1)$,
respectively. It leads to the equation in finite differences
\begin{equation}
u(x, t+1) - u(x, t-1) = a(u(x+1, t) - 2u(x, t) + u(x-1, t))
\end{equation}
where $a$ is a constant. As in the case of the forward difference
scheme, let $K$ be an inversive difference functional field with
basic set $\sigma = \{\alpha_{1}:f(x, t)\mapsto f(x+1, t),\,
\alpha_{2}:f(x, t)\mapsto f(x, t+1)\}$ ($f(x, t)\in K$) and let
$K\{y\}^{\ast}$ be the ring of $\sigma^{\ast}$-polynomials in one
$\sigma^{\ast}$-indeterminate $y$ over $K$ ($y$ is treated as the
unknown function $u(x, t)$;  we also assume that $a\in K$). Then the
equation (4.4) can be written as
\begin{equation}
a\alpha_{1}y + a\alpha_{1}^{-1}y - \alpha_{2}y  - \alpha_{2}^{-1}y -
2ay = 0.
\end{equation}
By Proposition 3.15, the characteristic set of the
$\sigma^{\ast}$-ideal generated by the $\sigma^{\ast}$-polynomial $B
= a\alpha_{1}y + a\alpha_{1}^{-1}y - \alpha_{2}y  - \alpha_{2}^{-1}y
- 2ay$ is $\{B, \alpha_{1}^{-1}B\}$. The leaders of $B$ and
$\alpha_{1}^{-1}B$ are $\alpha_{1}y$ and $\alpha_{1}^{-2}y$,
respectively. Now Theorem 2.10 shows that the strength of the
equation (4.5) is expressed by the dimension polynomial
$\omega_{E}(t)$ where $$E = \{(1, 0, 0, 0), (0, 0, 2, 0), (1, 0, 1,
0), (0, 1, 0, 1)\}\subseteq\mathbb{N}^{4}.$$ Applying formula (2.3)
we obtain that the strength of the equation (4.5), which expresses
the symmetric difference scheme for (4.1), is represented by the
$\sigma^{\ast}$-dimension polynomial $$\psi_{Symm}(t) = 4t.$$

\medskip

\large
\begin{center}
Strength of the Crank-Nicholson scheme
\end{center}
\normalsize
\medskip

The Crank-Nicholson scheme for the diffusion equation with the above
interpretation of the shifts of arguments as two automorphisms
$\alpha_{1}$ and $\alpha_{2}$ gives the algebraic difference
equation of the form
\begin{equation}
\alpha_{1}\alpha_{2}y + a_{1}\alpha_{1}^{-1}\alpha_{2}y +
a_{2}\alpha_{1}y + a_{3}\alpha_{2}y + a_{4}\alpha_{1}^{-1}y +
a_{5}=0
\end{equation}
where $a_{i}$ ($1\leq i\leq 5$) are constants. Applying Proposition
3.15, we obtain that the $\sigma^{\ast}$-polynomial\, $C =
\alpha_{1}\alpha_{2}y + a_{1}\alpha_{1}^{-1}\alpha_{2}y +
a_{2}\alpha_{1}y + a_{3}\alpha_{2}y + a_{4}\alpha_{1}^{-1}y + a_{5}$
in the left-hand side of the last equation generates a
$\sigma^{\ast}$-ideal of $K\{y\}^{\ast}$ whose characteristic set
consists of the $\sigma^{\ast}$-polynomials $C,\,
\alpha_{1}^{-1}C,\, \alpha_{2}^{-1}C$, and
$\alpha_{1}^{-1}\alpha_{2}^{-1}C$. Their leaders are
$\alpha_{1}\alpha_{2}y$, $\alpha_{1}^{-2}\alpha_{2}y$,
$\alpha_{1}\alpha_{2}^{-1}y$, and $\alpha_{1}^{-2}\alpha_{2}^{-1}y$,
respectively.

Applying theorem 2.10 and 2.8 to the set $\{(1, 1), (-2, 1),\\ (1,
-1), (-2, -1)\}\subseteq\mathbb{Z}^{2}$ we obtain that the strength
of the equation (4.6) is expressed by the dimension polynomial
$$\psi_{Crank-Nickolson}(t) = 6t-1 .$$ Thus, the symmetric
difference scheme for the diffusion equation has higher strength
(that is, smaller dimension polynomial) than the forward difference
scheme and the Crank-Nicholson scheme, so the symmetric scheme is
the best among these three schemes from the point of view of the
Einstein's strength.

\bigskip

{\bf 2.}\, {\bf Murray, Fisher, Burger and some other quasi-linear
reaction-diffusion equations.}
\smallskip

Proposition 3.14 allows us to compute the strength of
reaction-diffusion equations of the  form
\begin{equation}
{\frac{\partial u}{\partial t}}  - \frac{\partial^{2}u}{\partial
x^{2}} = P\left(u, {\frac{\partial u}{\partial x}}\right)
\end{equation}
where $u=u(x, t)$ is a function of space and time variables $x$ and
$t$, respectively, and $P\left(u, {\D\frac{\partial u}{\partial
x}}\right)$ is a nonlinear function of $u$ and ${\D\frac{\partial
u}{\partial x}}$. Such equations have recently attracted a lot of
attention in the context of chemical kinetics, mathematical biology
and turbulence. The following PDEs, that are particular cases of
equation (4.7), are in the core of the corresponding mathematical
models.

Murray equation \cite{Murray}:

\begin{equation}
\frac{\partial^{2}u}{\partial x^{2}} + \mu_{1}u{\frac{\partial
u}{\partial t}} + \mu_{2}u - \mu_{3}u^{2},\,\,\, (\mu_{1}, \mu_{2},
\mu_{3}\,\,\,\text{are constants)}.
\end{equation}

Burgers equation \cite{Burgers}:
\begin{equation}
\frac{\partial^{2}u}{\partial x^{2}} - u{\frac{\partial u}{\partial
x}} - {\frac{\partial u}{\partial t}} = 0.
\end{equation}

Fisher equation \cite{Fisher}:
\begin{equation}
\frac{\partial^{2}u}{\partial x^{2}} - u{\frac{\partial u}{\partial
t}} + u(1-u) = 0.
\end{equation}

Huxley equation \cite{Wazwaz1}:
\begin{equation}
\frac{\partial^{2}u}{\partial x^{2}} - u{\frac{\partial u}{\partial
t}} - u(k-u)(u-1) = 0, \,\,\, k\neq 0.
\end{equation}

Burgers-Fisher equation \cite{Wazwaz}:

\begin{equation}
\frac{\partial^{2}u}{\partial x^{2}} + u{\frac{\partial u}{\partial
x}} - {\frac{\partial u}{\partial t}} + u(1-u)= 0.
\end{equation}

Burgers-Huxley equation \cite{McRub}:

\begin{equation}
\frac{\partial^{2}u}{\partial x^{2}} + u{\frac{\partial u}{\partial
x}} - {\frac{\partial u}{\partial t}} + u(k-u)(u-1)= 0,\,\,\, k\neq
0.
\end{equation}

FitzHugh-Nagumo equation \cite{Wazwaz1}:

\begin{equation}
\frac{\partial^{2}u}{\partial x^{2}} + u{\frac{\partial u}{\partial
x}} + u(1-u)(a-u)= 0, \,\,\, a\neq 0.
\end{equation}

The last seven equations are of the form
\begin{equation}
\frac{\partial^{2}u}{\partial x^{2}} + (au + b){\frac{\partial
u}{\partial x}} + c{\frac{\partial u}{\partial t}} + F(u) = 0
\end{equation}
where $a, b, c$ are constants ($c\neq 0$, $ab\neq 0$) and $F(u)$ is
a polynomial in one variable $u$ with coefficients in the ground
functional field $K$. Therefore, the forward difference scheme for
equations (4.8) - (4.14) leads to algebraic difference equations of
the form

\begin{equation}
\alpha_{1}^{2}y + (ay+b-2)\alpha_{1}y + c\alpha_{2}y + G(y) = 0.
\end{equation}
(As before, we set $y=u$, denote the automorphisms of the ground
field $f(x, t)\mapsto f(x+1, t)$ and $f(x, t)\mapsto f(x, t+1)$ by
$\alpha_{1}$ and $\alpha_{2}$, respectively, and write the monomials
in the left-hand side of the equation in the decreasing order of
their highest terms. We also set $G(y) = F(y)-ay^{2}-(b+c-1)y$.)

Applying Propositions 3.14 and 3.15, we obtain that the
$\sigma^{\ast}$-polynomial $A = \alpha_{1}^{2}y +
(ay+b-2)\alpha_{1}y + c\alpha_{2}y + G(y)$ generates a prime
$\sigma^{\ast}$-ideal of $K\{y\}^{\ast}$ ($\sigma = \{\alpha_{1},
\alpha_{2}\}$). As in the case of equation (4.3), we obtain that the
characteristic set of the ideal $[A]^{\ast}$ consists of the
$\sigma^{\ast}$-polynomials $A$, $\alpha_{1}^{-1}A$,
$\alpha_{1}^{-1}\alpha_{2}^{-1}A$ and
$\alpha_{1}^{-2}\alpha_{2}^{-1}A$  with leaders $\alpha_{1}^{2}y$,
$\alpha_{1}^{-1}\alpha_{2}y$, $\alpha_{1}\alpha_{2}^{-1}y$ and
$\alpha_{1}^{-2}\alpha_{2}^{-1}y$, respectively. Therefore (as in
the case of equation (4.3)\,) the $\sigma^{\ast}$-dimension
polynomial that expresses the Einstein's strength of the forward
difference scheme for each of the equations (4.8) - (4.14) is equal
to the dimension polynomial of the set $\{(2, 0), (-1, 1), (1, -1),
(-2, -1)\}\subseteq\mathbb{Z}^{2}$, that is, $$\psi_{Forw}(t) =
5t.$$

\medskip

The symmetric difference scheme for equation (4.15) (and therefore
for each of the equations (4.8) - (4.14)\,) gives an algebraic
difference equation of the form

\begin{equation}
(ay+b+1)\alpha_{1}y + (1-ay-b)\alpha_{1}^{-1}y + c\alpha_{2}y -
c\alpha_{2}^{-1}y + F(y) = 0.
\end{equation}
(Recall that we replace $\D\frac{\partial u}{\partial x}$,
$\D\frac{\partial^{2}u}{\partial x^{2}}$ and $\D\frac{\partial
u}{\partial t}$ with $(\alpha_{1} + \alpha^{-1} -2)u$, $(\alpha_{1}
- \alpha_{1}^{-1})u$ and $(\alpha_{2} - \alpha_{2}^{-1})u$,
respectively.)

Equation (4.17) is not quasi-linear with respect to the standard
ranking described after Definition 3.4. However, if one considers a
similar ranking with $\alpha_{2} > \alpha_{1}$, then the
$\sigma^{\ast}$-polynomial $B = (ay+b+1)\alpha_{1}y +
(1-ay-b)\alpha_{1}^{-1}y + c\alpha_{2}y - c\alpha_{2}^{-1}y + F(y)$
in the left-hand side of (4.17) is a quasi-linear one with the
leader $\alpha_{2}y$. By Propositions 3.14 and 3.15, the
$\sigma^{\ast}$-polynomials $B$ and $\alpha_{2}^{-1}B$ form a
characteristic set of the prime $\sigma^{\ast}$-ideal $[B]^{\ast}$
of $K\{y\}^{\ast}$. Since their leaders are, respectively,
$\alpha_{2}y$ and $\alpha_{2}^{-2}y$, the Einstein's strength of the
symmetric difference scheme for each of the equations (4.8) - (4.14)
is expressed by the dimension polynomial $\psi_{Symm}(t)$ of the set
$\{(1, 0), (0, 2)\}\subseteq\mathbb{Z}^{2}$. As in the case of
equation (4.5), we obtain that  $$\psi_{Symm}(t) = 4t.$$

Thus, one should prefer the symmetric scheme to the forward one
while considering the Einstein's strength of these schemes for  PDEs
(4.8) - (4.14).

\bigskip

{\bf 3.}\, {\bf The mathematical model of chemical reaction kinetics
with the diffusion phenomena} is described by a system of partial
differential equations of the form

\begin{equation}
\left\{%
\begin{array}{l>{\raggedright}p{.5\textwidth}}%
{\D\frac{\partial u_{1}}{\partial t}} =  {\D\frac{\partial^{2} u_{1}}{\partial x^{2}}} - k_{1}u_{1}u_{2} + k_{2}u_{3},\\\\
{\D\frac{\partial u_{2}}{\partial t}} =  {\D\frac{\partial^{2} u_{2}}{\partial x^{2}}} - k_{1}u_{1}u_{2} + k_{2}u_{3},\\\\
{\D\frac{\partial u_{3}}{\partial t}} =  {\D\frac{\partial^{2} u_{3}}{\partial x^{2}}} + k_{1}u_{1}u_{2} - k_{2}u_{3}.
\end{array}%
\right. %
\end{equation}

(see \cite{Lim}). Setting $v_{1} = u_{1} - u_{2}$ and $v_{2} = u_{1}
+ u_{3}$, one can rewrite the system in the form

\begin{equation}
\left\{%
\begin{array}{l>{\raggedright}p{.5\textwidth}}%
{\D\frac{\partial v_{1}}{\partial t}} =  {\D\frac{\partial^{2} v_{1}}{\partial x^{2}}},\\\\
{\D\frac{\partial v_{2}}{\partial t}} =  {\D\frac{\partial^{2} v_{2}}{\partial x^{2}}}, \\\\
{\D\frac{\partial u_{1}}{\partial t}} =  {\D\frac{\partial^{2}
u_{1}}{\partial x^{2}}} - k_{1}u_{1}^{2} + k_{1}u_{1}v_{1} +
k_{2}v_{2} -k_{2}u_{1}.
\end{array}%
\right. %
\end{equation}

\medskip

The forward difference scheme leads to the following system of
algebraic difference equations with three
$\sigma^{\ast}$-indeterminates  $y_{1},\, y_{2}$ and $y_{3}$ (they
stand for $v_{1},\, v_{2}$ and $u_{1}$, respectively), where $\sigma
= \{\alpha_{1}:f(x, t)\mapsto f(x+1, t),\, \alpha_{2}:f(x, t)\mapsto
f(x, t+1)\}$ ($f(x, t)$ is an element of an inversive ground
functional field $K$).

\begin{equation}
\left\{%
\begin{array}{l>{\raggedright}p{.5\textwidth}}%
\alpha_{1}^{2}y_{1} - 2\alpha_{1}y_{1} - \alpha_{2}y_{1} + 2y_{1} =
0,\\\\
\alpha_{1}^{2}y_{2} - 2\alpha_{1}y_{2} - \alpha_{2}y_{2} + 2y_{2} =
0,\\\\
\alpha_{1}^{2}y_{3} - 2\alpha_{1}y_{3} - \alpha_{2}y_{3} +
k_{1}y_{1}y_{3} - k_{1}y_{3}^{2} + k_{2}y_{2} - k_{2}y_{3} = 0
\end{array}%
\right. %
\end{equation}

where $k_{1}, k_{2}, k_{3}$ are constants in $K$.

Let $A$, $B$, and $C$ be the $\sigma^{\ast}$-polynomials in the
left-hand sides of the  first, second and third equations of the
last system, respectively. Combining Proposition 2.4.9 of
\cite{Levin1} (that states that every linear $\sigma^{\ast}$-ideal
in a ring of $\sigma^{\ast}$-polynomials is prime) and our
Proposition 3.14 we obtain that the $\sigma^{\ast}$-ideal $P = [A,
B, C]^{\ast}$ of the ring $K\{y_{1}, y_{2}, y_{3}\}^{\ast}$ is
prime.

Let $\eta_{i}$ denote the canonical image of the
$\sigma^{\ast}$-indeterminate $y_{i}$ in the
$\sigma^{\ast}$-quotient field $L$ of $K\{y_{1}, y_{2},
y_{3}\}^{\ast}/P$ ($i=1, 2, 3$) and let $L_{1} =
K\langle\eta_{1}\rangle^{\ast}$ and $L_{2} =
L_{1}\langle\eta_{2}\rangle^{\ast}$ (so that $L =
L_{2}\langle\eta_{3}\rangle^{\ast}$).

For any $r\in\mathbb{N}$ and $\zeta\in L$, let $\Gamma(r)\zeta=
\{\gamma(\zeta)\,|\,\gamma\in\Gamma(r)\}$, $L_{1r} =
K(\Gamma(r)\eta_{1})$, $L_{2r} = K(\Gamma(r)\eta_{1}\bigcup
\Gamma(r)\eta_{2})$, and $L_{r} = K(\Gamma(r)\eta_{1}\bigcup
\Gamma(r)\eta_{2}\bigcup\Gamma(r)\eta_{3})$. Then Theorem 3.2 and
the obvious fact that the fields $K\langle\eta_{1}\rangle^{\ast}$,
$K\langle\eta_{2}\rangle^{\ast}$ and
$K\langle\eta_{3}\rangle^{\ast}$ are pairwise algebraically disjoint
over $K$ imply that there are numerical polynomials $\psi_{1}(t)$,
$\psi_{2}(t)$ and $\psi_{3}(t)$ such that $\psi_{1}(r) =
\trdeg_{K}L_{1r}$, $\psi_{2}(r) = \trdeg_{L_{1r}}L_{2r}$ and
$\psi_{3}(r) = \trdeg_{L_{2r}}L_{r}$ for all sufficiently large
$r\in\mathbb{N}$. Each of the polynomials $\psi_{i}(t)$ represents a
$\sigma^{\ast}$-dimension polynomial of the equation of the form
(4.3). Indeed, the first two equation of system (4.20) are similar
to equation (4.3) and the third equation of the system is a
quasi-linear $\sigma^{\ast}$-equation whose dimension polynomial,
according to Proposition 3.14, is equal to the dimension polynomial
of (4.3). Since $\trdeg_{K}L_{r} = \trdeg_{K}L_{1r} +
\trdeg_{L_{1r}}L_{2r} + \trdeg_{L_{2r}}L_{r}$, the
$\sigma^{\ast}$-dimension polynomial of the extension $L/K$ is equal
to  $\psi_{1}(t) + \psi_{2}(t) + \psi_{3}(t)$. Now the formula for
the strength of the forward difference scheme for the diffusion
equation (4.1) shows that the strength of the forward difference
scheme for system (4.19) is expressed with the polynomial
$$\psi_{Forw}(t) = 15t.$$

Using the above arguments and the results for difference schemes for
equation (4.1), we obtain that the strengths of the symmetric and
Crank-Nicholson schemes for (4.19) are expressed with the
polynomials
$$\psi_{Symm}(t) = 12t\,\,\,
\text{and}\,\,\, \psi_{Crank-Nicholson}(t) = 18t - 3,$$
respectively. Therefore, in our case, as in the case of equation
(4.1), the symmetric scheme for system (4.19) is the best among
these three schemes.

\bigskip

{\bf 4.}\, {\bf The mass balance equations of chromatography} for
the case of $N$ components in a column slice (see, \cite[p. 24]{GSK}
form the following system of PDEs:
\begin{equation}
\left\{%
\begin{array}{l>{\raggedright}p{.5\textwidth}}%
{\D\frac{\partial C_{1}}{\partial t}} + F{\D\frac{\partial C_{s,1}}{\partial t}} + u{\D\frac{\partial C_{1}}{\partial z}}  = D_{L,1}{\D\frac{\partial^{2}C_{1}}{\partial z^{2}}},\\\\

\dots\\\\

{\D\frac{\partial C_{N}}{\partial t}} + F{\D\frac{\partial C_{s,N}}{\partial t}} + u{\D\frac{\partial C_{N}}{\partial z}}  = D_{L,N}{\D\frac{\partial^{2}C_{N}}{\partial z^{2}}}
\end{array}%
\right. %
\end{equation}
where  $C_{i}$ and $C_{s,i}$ ($1\leq i\leq N$) are the
concentrations of the individual components in the mobile phase and
in the stationary phase, respectively ($u$ and $D_{L,i}$, $1\leq
i\leq N$, are constants along the column).

Let $K$ be the inversive difference ($\sigma^{\ast}$-) functional
field over which we consider the system (''ground field''). As
before, $\sigma = \{\alpha_{1}, \alpha_{2}\}$ where for any function
$f(z)\in K$, $\alpha_{1}:f(z, t)\mapsto f(z+1, t)$ and
$\alpha_{2}:f(z, t)\mapsto f(z, t+1)$.

The {\bf forward difference scheme} leads to a system of algebraic
difference equations

\begin{equation}
\left\{%
\begin{array}{l>{\raggedright}p{.5\textwidth}}%
D_{L,1}\alpha_{1}^{2}y_{1} + a_{1}\alpha_{1}y_{1} -
F\alpha_{2}y_{2} - \alpha_{2}y_{1} + Fy_{2} + b_{1}y_{1} = 0,\\\\

\dots\\\\

D_{L,N}\alpha_{1}^{2}y_{2N-1} + a_{N}\alpha_{1}y_{2N-1} -
F\alpha_{2}y_{2N} - \alpha_{2}y_{2N-1} + Fy_{2N} + b_{N} y_{2N-1} =
0
\end{array}%
\right. %
\end{equation}
where $y_{2i-1}$ and $y_{2i}$ ($i=1,\dots, N$) stand for $C_{i}$ and
$C_{s,i}$, respectively, $a_{i} = - u-2D_{L,i}$, and $b_{i} =
-D_{L,i} - u - 1$. (The terms in the $\sigma^{\ast}$-polynomials in
the left-hand sides of the last system are arranged in the
decreasing order with respect to the standard ranking.)

Let $$A_{i} = D_{L,i}\alpha_{1}^{2}y_{2i-1} +
a_{i}\alpha_{1}y_{2i-1} - F\alpha_{2}y_{2i} - \alpha_{2}y_{2i-1} +
Fy_{2i} + b_{i} y_{2i-1}$$ ($1\leq i\leq N$), let $P$ denote the
linear (and therefore prime) $\sigma^{\ast}$-ideal $[A_{1},\dots,
A_{N}]^{\ast}$ of the ring $R = K\{y_{1},\dots, y_{2N}\}^{\ast}$ and
let $\eta_{j}$ denote the canonical image of the
$\sigma^{\ast}$-indeterminate $y_{j}$ in the factor ring $R/P$, so
that the quotient field $L$ of $R/P$ can be represented as $L =
K\langle\eta_{1},\dots, \eta_{2N}\rangle^{\ast}$. Then one can
repeat arguments of the analysis of system (4.20) and obtain that
the $\sigma^{\ast}$-dimension polynomial of system (4.22) is the sum
of $N$ equal $\sigma^{\ast}$-dimension polynomials associated with
the $\sigma^{\ast}$-field extensions of $K$ defined by equations
$A_{i} = 0$ ($i=1,\dots, N$).

In order to find the $\sigma^{\ast}$-dimension polynomial associated
with the equation $A_{1}=0$ (and therefore with any equation $A_{i}
= 0$, $1\leq i\leq N$), we can apply Proposition 3.15 and obtain
that $\sigma^{\ast}$-polynomials $A_{1},\, \alpha_{1}^{-1}A_{1},\,
\alpha_{1}^{-1}\alpha_{2}^{-1}A_{1}$ and
$\alpha_{1}^{-2}\alpha_{2}^{-1}A_{1}$ with leaders
$\alpha_{1}^{2}y_{1}$, $\alpha_{1}^{-1}\alpha_{2}y_{2}$,
$\alpha_{1}\alpha_{2}^{-1}y_{1}$ and
$\alpha_{1}^{-2}\alpha_{2}^{-1}y_{2}$, respectively, form a
characteristic set of the prime $\sigma^{\ast}$-ideal
$[A_{1}]^{\ast}$ of the ring $K\{y_{1}, y_{2}\}^{\ast}$. It follows
that the desired $\sigma^{\ast}$-dimension polynomial is the sum of
the dimension polynomials of subsets $ \{(2, 0), (1, -1)\}$ and
$\{(-1, 1), (-2, -1)\}$ of $\mathbb{Z}^{2}$. By Theorem 2.10, these
polynomials are equal, respectively,  to the dimension polynomials
$\omega_{E_{1}}(t)$ and $\omega_{E_{2}}(t)$ of subsets $$E_{1} =
\{(2, 0, 0, 0), (1, 0, 0, 1), (1, 0, 1, 0), (0, 1, 0, 1)\}$$ and
$$E_{2} = \{(0, 1, 1, 0), (0, 0, 2, 1), (1, 0, 1, 0),
(0, 1, 0, 1)\}$$ of $\mathbb{N}^{2}$. Applying formula (2.3) in
Theorem 2.8 we get $\omega_{E_{1}}(t) = t^{2}+3t+1$ and
$\omega_{E_{2}}(t)= t^{2}+4t$. It follows that the
$\sigma^{\ast}$-dimension polynomial of the equation $A_{1} = 0$ (as
well as of every equation $A_{i} = 0$, $1\leq i\leq N$) is
$\omega_{E_{1}}(t) + \omega_{E_{2}}(t) = 2t^{2}+7t+1$. Thus, the
$\sigma^{\ast}$-dimension polynomial that expresses the strength of
the forward difference scheme for system (4.21) is as follows:
\begin{equation}
\psi_{Forw}(t) = 2t^{2}+7t+1.
\end{equation}

Applying the {symmetric difference scheme} to system (4.21) we
obtain $N$ equations of the form
$$D_{L,i}(\alpha_{2i-1}-2+\alpha^{-1})y_{2i-1} -
(\alpha_{2}-\alpha_{2}^{-1})y_{2i-1} -
F(\alpha_{2}-\alpha_{2}^{-1})y_{2i} -
u(\alpha_{1}-\alpha_{1}^{-1})y_{2i-1} = 0$$ or

\begin{equation}
(u-D_{L,i})\alpha_{1}y_{2i-1} - (u+D_{L,i})\alpha_{1}^{-1}y_{2i-1} +
F\alpha_{2}y_{2i} - F\alpha_{2}^{-1}y_{2i} + 2D_{L,i}y_{2i-1} = 0.
\end{equation}
with $\sigma^{\ast}$-indeterminates in the ring $K\{y_{1},\dots,
y_{2N}\}^{\ast}$.

Denoting the $\sigma^{\ast}$-polynomial in the left-hand side of the
last equation by $B$, one can easily see that $\{B,
\alpha_{1}^{-1}B\}$ is a characteristic set of the linear
$\sigma^{\ast}$-ideal $[B]^{\ast}$ in $K\{y_{2i-1}, y_{2i}\}^{\ast}$
($1\leq i\leq N$). The corresponding leaders are
$\alpha_{1}y_{2i-1}$ and $\alpha_{1}^{-2}y_{2i-1}$.  It follows that
the strength of the system of difference equations obtained from the
symmetric difference scheme for system (4.21) is expressed with the
sum of the dimension polynomial of the set $\{(1, 0), (-2,
0)\}\subseteq\mathbb{Z}^{2}$ (this set corresponds to the leaders
containing $y_{2i-1}$) and the dimension polynomial of the empty set
(that corresponds to the leaders containing $y_{2i}$). Using the
last part of Theorem 2.9 and the result of the computation of the
$\sigma^{\ast}$-dimension polynomial of equation (4.5) we obtain
that the strength of the symmetric difference scheme for system
(4.21) is represented by the polynomial
\begin{equation}
\psi_{Symm}(t) = \sum_{i=0}^{2}(-1)^{2-i}2^{i}{2\choose
i}{{t+i}\choose i} + 4t = 2t^{2}+6t+1.
\end{equation}
As we see, in this case one should also prefer the symmetric scheme
to the forward one; however, there is quite a small difference
between these two schemes.

\bigskip

This work was supported by the NSF grant CCF-1714425.

\end{document}